\documentclass[11pt]{article}

\usepackage{latexsym}
\usepackage{amssymb}
\usepackage{amsthm}
\usepackage{amscd}
\usepackage{amsmath}
\usepackage{mathrsfs}
\usepackage[all]{xy}
\usepackage{hyperref} 
\usepackage[usenames,dvipsnames]{color}
\usepackage{graphicx,eepic}

\usepackage{color}

\theoremstyle{definition}
\newtheorem* {theorem*}{Theorem}

\theoremstyle{definition}
\newtheorem* {conjecture*}{Conjecture}

\theoremstyle{definition}
\newtheorem{theorem}{Theorem}[section]

\theoremstyle{definition}

\theoremstyle{definition}
\newtheorem{observation}{Observation}[section]

\theoremstyle{definition}
\newtheorem{case}{Case}

\newtheorem{lemma}{Lemma}[section]
\theoremstyle{definition}

\theoremstyle{definition}

\theoremstyle{definition}

\newtheorem{conjecture}{Conjecture}[section]
\newtheorem{proposition}{Proposition}[section]

\newtheorem* {remark}{Remark}
\newtheorem{example}{Example}[section]
\theoremstyle{definition}

\theoremstyle{definition}
\newtheorem* {remarks}{Remarks}

\xyoption{dvips}

\usepackage{fullpage}

\numberwithin{equation}{section}

\def\({\left(}
\def\){\right)}
   \newcommand{\FF}{\mathbb{F}}  \newcommand{\CC}{\mathbb{C}}  \newcommand{\QQ}{\mathbb{Q}}   \newcommand{\cP}{\mathcal{P}} \newcommand{\cA}{\mathcal{A}}
   
  \newcommand{\cS}{\mathcal{S}} 
\newcommand{\cJ}{\mathcal{J}} 
\newcommand{\cC}{\mathcal{C}}
\newcommand{\cB}{\mathcal{B}}
\newcommand{\cD}{\mathcal{D}}
\newcommand{\cZ}{\mathcal{Z}}

\def\tr{\mathrm{tr}}    \def\ZZ{\mathbb{Z}}  \def\X{\mathfrak{X}}  \def\Ind{\mathrm{Ind}} \def\GL{\mathrm{GL}}   \def\Res{\mathrm{Res}}    \def\spanning{\textnormal{-span}}   \def\cB{\mathcal{B}}
\def\Irr{\mathrm{Irr}}  \def\wt{\widetilde}

   \newcommand{\fkn}{\mathfrak{n}}
\newcommand{\h}{\mathfrak{h}}

\def\X{\mathcal{X}}

\newcommand{\fka}{\mathfrak{a}}

\newcommand{\leftexp}[2]{{\vphantom{#2}}^{#1}{#2}}

\def\fk{\mathfrak}

\def\barr{\begin{array}}
\def\earr{\end{array}}
\def\ba{\begin{aligned}}
\def\ea{\end{aligned}}
\def\be{\begin{equation}}
\def\ee{\end{equation}}
\def\cS{\mathcal{S}}

\def\ol{\widehat}
\def\olfkl{\ol{\fk l}}
\def\olfks{\ol{\fk s}}
\def\oll{\ol{L}}
\def\ols{\ol{S}}

\def\UT{\mathrm{UT}}
\def\fkt{\fk{u}}
\def\logpsi{\psi^{\exp}}
\def\tobedecided{}

\def\exp{\mathrm{Exp}}

\makeatletter
\renewcommand{\@makefnmark}{\mbox{\textsuperscript{}}}
\makeatother

\allowdisplaybreaks[1]

\UseCrayolaColors

\begin{document}
\title{Exotic characters of unitriangular matrix groups}
\author{Eric Marberg\footnote{This research was conducted with government support under
the Department of Defense, Air Force Office of Scientific Research, National Defense Science
and Engineering Graduate (NDSEG) Fellowship, 32 CFR 168a.} \\ Department of Mathematics \\ Massachusetts Institute of Technology \\ \tt{emarberg@math.mit.edu}}
\date{}

\maketitle


\begin{abstract}
Let $\UT_n(q)$ denote the unitriangular group of unipotent $n\times n$ upper triangular matrices over a finite field with cardinality $q$ and prime characteristic $p$.  It has been known for some time that when $p$ is fixed and $n$ is sufficiently large, $\UT_n(q)$ has ``exotic'' irreducible characters taking values outside the cyclotomic field $\QQ(\zeta_p)$.  However, all proofs of this fact to date have been both non-constructive and computer dependent. 
In the preliminary work \cite{supp0}, we defined a family of orthogonal characters decomposing the supercharacters of an arbitrary algebra group.
By applying this construction to the unitriangular group, we are able to
derive by hand
an explicit description of a family of characters of $\UT_n(q)$ taking values in arbitrarily large cyclotomic fields.
In particular, we prove that if $r$ is a positive integer power of $p$ and $n>6r$, then $\UT_n(q)$ has an irreducible character of degree $q^{5r^2-2r}$ which takes values outside $\QQ(\zeta_{pr})$.  By the same techniques, we are also able to construct explicit Kirillov functions which fail to be characters  of $\UT_n(q)$ when $n>12$ and $q$ is arbitrary.
\end{abstract}

\section{Introduction}

Let $\FF_q$ be a finite field with $q$ elements and write $\UT_n(q)$ to the denote the unitriangular group of $n\times n$ upper triangular matrices over $\FF_q$ with all diagonal entries equal to 1.  This is a $p$-Sylow subgroup of the general linear group $\GL(n,\FF_q)$, where $p>0$ is the characteristic of $\FF_q$.  Researchers have known for some time that for large values of $n$, there exist ``exotic'' irreducible characters of $\UT_n(q)$ which have values outside the cyclotomic field $\QQ(\zeta_p)$, where $\zeta_p$ is a primitive $p$th root of unity.  However, proofs of this fact to date have been largely both nonconstructive and computer dependent. 

 For example, Isaacs and Karaguezian showed indirectly that $\UT_n(2)$ has a nonreal character for $n>12$ by writing down a matrix in $\UT_{13}(2)$ not conjugate to its inverse \cite{IK1,IK2}.  The same authors later gave a different computational proof by implementing algorithms to compute the character degrees  and involutions of $\UT_n(2)$ \cite{IK05}.  More generally,  Vera-L\'opez and Arregi have shown through some detailed calculations involving the help of a computer algebra system that for $q$ prime,  $\UT_n(q)$ has an element not conjugate to its $(q+1)$th power for all $n>6q$ \cite{VeraLopez2004}.  By standard results in character theory (see \cite[Chapter 6]{Isaacs}), these conjugacy properties imply the existence of our exotic characters, but do not shed much light on any of their attributes.

One obstacle to providing more constructive proofs of these facts comes from our incomplete understanding of the representations of $\UT_n(q)$ when the characteristic of $\FF_q$ is small compared to $n$.   
Combatting this problem, we describe in \cite{supp0}  a generic method of constructing  characters of algebra groups such as the untriangular group, building on combined work of Andr\'e \cite{Andre1}, Yan \cite{Yan},  and Diaconis and Isaacs \cite{DI}.  
  The primary intent of this work is to use this new construction to identify certain irreducible characters of $\UT_n(q)$ and then to prove by hand that these characters take values outside various cyclotomic fields.  Our techniques shed light on how ``exotic'' characters of $\UT_n(q)$ arise, and provide a computer independent proof of the following theorem.

\begin{theorem*} Suppose $\FF_q$ has characteristic $p$ and let $r=p^e$ for an integer $e>0$.  If $n > 6r$, then  $\UT_n(q)$ has an irreducible character of degree $q^{5r^2-2r}$ whose set of values is contained in $\QQ(\zeta_{pr})$ but not $\QQ(\zeta_r)$.  
\end{theorem*}

\def\Ad{\mathrm{Ad}}

These methods have another application concerning the Kirillov functions of $\UT_n(q)$.
Let $\fkt_n(q)$ denote the algebra of $n\times n$ upper triangular matrices over $\FF_q$ with all diagonal entries equal to 0.
There is a coadjoint action of $\UT_n(q)$ on the irreducible characters of $\fkt_n(q)$ viewed as an abelian group, given by $g : \vartheta \mapsto \vartheta \circ \Ad(g)^{-1}$ where $\Ad(g)(X) = gXg^{-1}$ for $g \in \UT_n(q)$ and $X \in \fkt_n(q)$.
If $\Omega$ is a coadjoint orbit, then the corresponding \emph{Kirillov function} $\psi : \UT_n(q) \to \QQ(\zeta_p)$ is the complex-valued function 
\[ \psi(g) =|\Omega|^{-1/2} \sum_{\vartheta \in \Omega} \vartheta(g-1),\qquad\text{for }g\in \UT_n(q).
\] Kirillov \cite{K} conjectured that these functions comprise all the irreducible characters of $\UT_n(q)$, and we observed in \cite{supp0} that a recent calculation of Evseev \cite{E} shows non-constructively that this conjecture holds if and only if $n\leq 12$.  Here we will be able to give a constructive proof of the ``only if'' direction of this result;
 in particular we shall identify a Kirillov function of degree $q^{16}$ which is not a character of $\UT_n(q)$ when $n> 12$.

Our methods also shed some light on a different type of Kirillov function.  If $\exp : \fkt_n(q) \to \UT_n(q)$ denotes the truncated exponential map 
$\exp(X) = 1+X + \frac{1}{2}X^2 + \dots + \frac{1}{(p-1)!} X^{p-1}$, 
then an \emph{exponential Kirillov function} $\psi^\exp : \UT_n(q) \to  \QQ(\zeta_p)$ is a function defined by
\[ \psi^\exp\( \exp(X) \) = \psi(1+X),\qquad\text{for }X \in \fkt_n(q)\] for some Kirillov function $\psi$.
Sangroniz \cite{Sangroniz} has shown that every irreducible character of $\UT_n(q)$ is an exponential Kirillov function if $n < 2p$, and an indirect consequence of Vera-Lopez and Arregi's work \cite{VeraLopez2004}
is that there exist exponential Kirillov functions which are not characters when $q$ is prime and $n>6q$.  
We extend this result to arbitrary finite fields by identifying an exponential Kirillov function of degree $q^{5p^2-p}$ which is not a character of $\UT_n(q)$ when $n>6p$.  
It seems conceivable that our construction of this function is the simplest possible, so we conjecture the following.

\begin{conjecture*} If $p>0$ is the characteristic of $\FF_q$, then 
the irreducible characters and exponential Kirillov functions of $\UT_n(q)$ coincide if and only if $n\leq 6p$.  
\end{conjecture*}

While one can verify this statement using computers  for $p=2$ (see \cite{E}),   examining even the case $p=3$ exits the realm of currently feasible calculations.  Also, there is no proof to date that the characters of $\UT_n(q)$ are necessarily $\QQ(\zeta_p)$-valued if $n\leq 6p$.  Nevertheless, it seems possible, just from the experience of proving the results herein,
 that $n=6p$ is the ``breaking point'' after which the algebra groups $\UT_n(q)$ can manifest  irreducible characters which are not exponential Kirillov functions.  We mention that computer experiments suggest that one might be able to use a clever combinatorial argument$-$combining Lemmas 9, 10, and 11 in \cite{Sangroniz} with Theorem \ref{structural} and Lemma \ref{monomial} below 
 to show that the ``if'' direction of this conjecture holds for $n\leq 3p$.  Improving this bound to $n\leq 6p$ probably will require more robust techniques, however.

\section{Preliminaries}\label{prelim-sect}

Here we briefly  establish our notational conventions, then describe several different functions on $\UT_n(q)$ which will be of interest in our later computations.    We present this material from the more general standpoint of algebra groups, although our applications shall primarily concern $\UT_n(q)$.

\subsection{Conventions and notation}

Given a finite group $G$, we let $\langle\cdot,\cdot\rangle_G$ denote the standard inner product on the complex vector space of functions $G \to \CC$ defined by 
$ \langle f,g\rangle_G = \frac{1}{|G|} \sum_{x \in G} f(x) \overline{g(x)}$, 
and write $\Irr(G)$ to denote the set of complex irreducible characters of $G$, or equivalently the set of characters $\chi$ of $G$ with $\langle \chi,\chi \rangle_G = 1$.  A function $ G\to \CC$ is then a character if and only if it is a nonzero sum of irreducible characters with nonnegative integer coefficients.

 If $f: S \to T$ is a map and $S'\subset S$, then we write $f\downarrow S'$ to denote the restricted map $S'\to T$.  For functions on groups, we may also write $\Res_H^G(\chi) = \chi\downarrow H$ to denote the restriction of $\chi : G\to \CC$ to a subgroup $H$.  
If $\chi$ is any complex valued function whose domain includes the subgroup $H\subset G$, then we define the induced function $\Ind_{H}^G(\chi):G \to \CC$ by the formula
 \be\label{frob} \Ind_H^G(\chi)(g) = \frac{1}{|H|} \sum_{\substack { x \in G  \\ xgx^{-1} \in H}} \chi(xgx^{-1}),\qquad\text{for }g\in G.\ee  We recall that restriction takes characters of $G$ to characters of $H$ and induction take characters of $H$ to characters of $G$.

Throughout, $q>1$ is some fixed prime power and $\FF_q$ is a finite field with $q$ elements.  We write $\FF_q^+$ to denote the additive group of the field and $\FF_q^\times$ to denote the multiplicative group of nonzero elements.  For any positive integer $r$ we let $\zeta_r$ denote the primitive $r$th root of unity $\zeta_r = e^{2\pi i / r}$, and write $\QQ(\zeta_r)$ to denote the $r$th cyclotomic field, given by adjoining $\zeta_r$ to $\QQ$.  For integers $m,n$ we let 
\[ [m,n] = \{ t \in \ZZ : m\leq t \leq n\}\qquad\text{and}\qquad [n] = [1,n] = \{ 1,2,\dots,n\}.\]
%
%
 Given integers $1\leq i<j \leq n$ we let 
\[ \ba e_{ij} & = \text{the matrix in $\fkt_n(q)$ with 1 in position $(i,j)$ and zeros elsewhere,}\\
e_{ij}^* &=\text{the $\FF_q$-linear map $\fkt_n(q)\to \FF_q$ given by $e_{ij}^*(X) = X_{ij}$.}\ea\]
These matrices and maps are then dual bases of $\fkt_n(q)$ and its dual space $\fkt_n(q)^*$.

\subsection{Algebra groups}
\label{alg}

Let $\fkn$ be a (finite-dimensional, associative) nilpotent $\FF_q$-algebra, and $\fkn^*$ its dual space of $\FF_q$-linear maps $\fkn \to \FF_q$.  Write $G = 1+\fkn$ to denote the corresponding \emph{algebra group}; this is the set of formal sums $1+X $ with $X \in \fkn$, made a group via the multiplication \[(1+X)(1+Y) = 1+X+Y+XY.\]  As prototypical examples, we take $\fkn$ to be the algebra $\fkt_n(q)$ of strictly upper triangular $n\times n$ matrices over $\FF_q$ and $G$ to be the \emph{unitriangular group} $\UT_n(q) = 1 + \fkt_n(q)$.  A considerable literature exists on algebra groups and their representations, of which the reader might take \cite{I95} as a starting point.

We call a subgroup of $G=1+\fkn$ of the form $H = 1+\h$ where $\h \subset \fkn$ is a subalgebra an \emph{algebra subgroup}.  
If $\h \subset \fkn$ is a two-sided ideal then $H$ is a normal algebra subgroup of $G$, and the map $gH \mapsto 1+(X+\h)$ for  $g=1+X \in G$ gives an isomorphism $G/H \cong 1 + \fkn /\h$.  In practice we shall usually identify the quotient $G/H$ with the algebra group $1 + \fkn /\h$ by way of this canonical map.

\def\osdef{\overset{\mathrm{def}}}

\subsection{Kirillov functions $\psi_\lambda$ and $\logpsi_\lambda$}

For the duration of this work, $\theta : \FF_q^+\to \CC^\times$ denotes a fixed, nontrivial homomorphism from the additive group of $\FF_q$ to the multiplicative group of nonzero complex numbers.  Observe that $\theta$ takes values in the cyclotomic field $\QQ(\zeta_p)$, where $p>0$ is the characteristic of $\FF_q$.
For each $\lambda \in \fkn^*$, we define $\theta_\lambda : G \to \QQ(\zeta_p)$ as the function with
\[ \theta_\lambda(g) = \theta\circ \lambda(g-1),\qquad\text{for }g \in G.\]  
The maps $\theta\circ \lambda : \fkn \to \CC$ are the distinct irreducible characters of the abelian group $\fkn$, and from this it follows that the functions $\theta_\lambda : G\to \CC$ are an orthonormal basis (with respect to $\langle\cdot,\cdot\rangle_G$) for all functions on the group.   

The most generic methods we have at our disposal for constructing characters of algebra groups involve summing the functions $\theta_\lambda$ over orbits in $\fkn^*$ under an appropriate action of $G$.
Kirillov functions provide perhaps the natural example of such a construction.  Their definition relies on the \emph{coadjoint} action of $G$ on $\fkn^*$, by which we mean the right action $(\lambda,g) \mapsto \lambda^{g}$ where we define
\[ 
 \lambda^g(X) = \lambda(g X g^{-1}),\qquad\text{for } \lambda \in \fkn^*,\ g\in G,\ X \in \fkn.\]  Denote  the coadjoint orbit of $\lambda \in \fkn^*$ by $\lambda ^G$.    
The \emph{Kirillov function} $\psi_\lambda$ is then the map  $G \to \QQ(\zeta_p)$ defined by 
\be\label{kirillov-def} \psi_\lambda = \frac{1}{\sqrt{|\lambda^G|}} \sum_{\mu \in \lambda^G} \theta_\mu.\ee The size of $\lambda^G$ is a power of $q$ to an even integer \cite[Lemma 4.4]{DI} and so $\psi_\lambda(1) = \sqrt{|\lambda^G|}$ is a nonnegative integer power of $q$. 
We have $\psi_\lambda = \psi_\mu$ if and only if $\mu \in \lambda^G$, and the distinct Kirillov functions on $G$ form an orthonormal basis (with respect to $\langle\cdot,\cdot\rangle_G$) for the class functions on the group.  Kirillov functions may fail to be characters, however, and one purpose of this work is to demonstrate how one can use the results in \cite{supp0} to prove this failure directly.

Our definition of a Kirillov function attempts to attach a coadjoint orbit in $\fkn^*$ to an irreducible character of $G$ 
 by way of the bijection $\fkn \to G$ given by $X \mapsto 1+X$.  This is Kirillov's orbit method in the context of finite groups as described in \cite{K}.  In practice, one is more successful in developing a correspondence between coadjoint orbits and irreducible characters if a different bijection $\fkn \to G$ is used.
 Specifically, 
let $\tobedecided \exp : \fkn \to G$ be the truncated exponential map \[\ba & \barr{l}\tobedecided \exp (X) = 1 + X + \frac{1}{2}X^2 + \frac{1}{6}X^3 + \dots + \frac{1}{(p-1)!}X^{p-1}. 
\earr\ea\]  This map is always a bijection (as is any polynomial map with constant term one), so we may define the \emph{exponential Kirillov function} $\logpsi_\lambda : G\to \QQ(\zeta_p)$ by 
\[ \logpsi_\lambda\(\exp(X)\) = \psi_\lambda(1+X),\qquad\text{for }X \in \fkn.\]
Exponential Kirillov functions have all the same properties are ordinary Kirillov functions; in particular,  they also form an orthonormal basis for the class functions on $G$ and coincide with ordinary Kirillov functions in characteristic two.  They are more often irreducible, however.  In particular, if $p$ is the characteristic of $\FF_q$ then
\begin{enumerate}
\item[(1)] $\Irr(G) =\left \{ \logpsi_\lambda : \lambda \in \fkn^*\right\}$ if $\fkn^p=0$ \cite[Corollary 3]{Sangroniz}.
\item[(2)] More strongly, $\Irr(\UT_n(q)) = \left\{ \logpsi_\lambda : \lambda \in \fkt_n(q)^*\right\}$ if $n < 2p$ \cite[Corollary 12]{Sangroniz}.  
\end{enumerate}
By contrast, $\Irr(\UT_n(q)) =\left \{ \psi_\lambda : \lambda \in \fkt_n(q)^*\right\}$ if and only if $n \leq 12$ \cite[Theorem 4.1]{supp0}.  As mentioned in the introduction, the bound $n<2p$ in (2) seems unlikely to be optimal.  In Section \ref{constructions} we will be able to identify exponential Kirillov functions which are not characters when $n>6p$; 
 our construction seems ``morally'' like the simplest possible, and this motivates the following conjecture.

\begin{conjecture}
$\Irr(\UT_n(q)) = \left\{ \logpsi_\lambda : \lambda \in \fkt_n(q)^*\right\}$ if and only if $n \leq 6p$.
\end{conjecture}

\subsection{Supercharacters $\chi_\lambda$}

While Kirillov functions provide an accessible orthonormal basis for the class functions of an algebra group,  supercharacters  alternatively provide an accessible family of orthogonal characters.  
Andr\'e \cite{Andre1} first defined these characters in the special case $G=\UT_n(q)$ as a practical substitute for the group's unknown irreducible characters.  Several years later, Yan \cite{Yan} showed how one could replace Andr\'e's definition with a more elementary construction, which Diaconis and Isaacs \cite{DI} subsequently generalized to algebra groups.

We define the supercharacters of $G=1+\fkn$ in a way analogous to Kirillov functions, but using left and right actions of $G$ on $\fkn^*$ in place of the coadjoint action.
In detail, the group $G$ acts on the left and right on $\fkn$  by multiplication, and on 
 $\fkn^*$ by $(g , \lambda) \mapsto g\lambda$ and $( \lambda,g) \mapsto \lambda g$ where we define
 \[ g\lambda(X) = \lambda(g^{-1}X)\qquad\text{and}\qquad \lambda g(X) = \lambda(Xg^{-1}),\qquad\text{for }\lambda \in \fkn^*,\ g \in G,\ X \in \fkn.\]
 These actions commute, in the sense that $(g\lambda) h = g(\lambda h)$ for $g,h \in G$, so there is no ambiguity in removing all parentheses and writing expressions like $g\lambda h$.   We denote the left, right, and two-sided orbits of $\lambda \in \fkn^*$ by  $G\lambda$, $\lambda G$, and $G\lambda G$ .
Notably, $G\lambda$ and $\lambda G$ have the same cardinality and $|G\lambda G| = \frac{|G\lambda||\lambda G|}{|G\lambda \cap \lambda G|}$ \cite[Lemmas 3.1 and 4.2]{DI}.
 
The \emph{supercharacter} $\chi_\lambda$ is then the function  $G \to \QQ(\zeta_p)$ defined by 
\be\label{superchar-def} 
 \chi_\lambda = \frac{|G\lambda|}{|G\lambda G|} \sum_{\mu \in G \lambda G} \theta_\mu.
\ee 
Supercharacters are always characers but often reducible.  We have $\chi_\lambda = \chi_\mu$ if and only if $\mu \in G\lambda G$, and every irreducible character of $G$ appears as a constituent of a unique supercharacter.  The orthogonality of the functions $\theta_\mu$ implies that \[\langle \chi_\lambda,\chi_\mu\rangle_G = \left\{\barr{ll} |G\lambda \cap \lambda G|, &\text{if }\mu \in G \lambda G , \\
0,&\text{otherwise,}\earr\right.
\qquad\text{for }\lambda,\mu \in \fkn^*.\] 
Furthermore, $\frac{|G\lambda|}{|G\lambda \cap \lambda G|} \chi_\lambda$ is the character of a two-sided ideal in $\CC G$.

The supercharacters of an algebra group rarely give us all irreducible charactrs.  For example, 
every irreducible character of $\UT_n(q)$ is a supercharacter only when $n\leq 3$.  Supercharacters nevertheless provide a useful starting point for constructing $\Irr(G)$.
In addition,
there are many interesting connections between the supercharacters of $\UT_n(q)$ and the combinatorics of set partitions, which we will touch upon briefly in Section \ref{pattern}.

\subsection{Supercharacter constituents $\xi_\lambda$}
\label{xi}

\def\rad{\mathrm{rad}}
\def\radl{\ker_{\mathrm{L}}}
\def\radr{\ker_{\mathrm{R}}}

Here we describe the character construction given in \cite{supp0}.  This is a process for decomposing supercharacters into smaller constituents $\xi_\lambda$, obtained by inducing linear characters of certain algebra subgroups.  The characters $\xi_\lambda$ will not give our ``exotic'' characters of $\UT_n(q)$ directly, but will possess in special cases a transparent decomposition into such characters.

As with Kirillov functions and supercharacters, the characters of present interest are indexed by elements of the dual space $\fkn^*$.  The relevant definition is somewhat more involved, however, and goes as follows.
  For each $\lambda \in \fkn^*$, define two sequences of subspaces $\fk l_\lambda^i, \fk s_\lambda^i \subset \fkn$ for $i\geq 0$ by the inductive formulas 
\[ \ba \fk l_\lambda ^0 &= 0, \\ 
\fk s_\lambda^0 &= \fkn, \ea \qquad \text{and}\qquad 
\ba \fk l_\lambda^{i+1} & = \left\{ X \in \fk s_\lambda^i : \lambda(XY) = 0 \text{ for all }Y \in \fk s_\lambda^i \right\}, \\ 
\fk s_\lambda^{i+1} & =\left \{ X \in \fk s_\lambda^i : \lambda(XY) = 0 \text{ for all }Y \in \fk l_\lambda^{i+1}\right \}. \ea\]
If $B_\lambda : \fkn\times \fkn \to \FF_q$ denotes the bilinear form $(X,Y) \mapsto \lambda(XY)$, then we may alternatively define the subspaces $\fk l_\lambda^i$, $\fk s_\lambda^i$ for $i>0$ by
\[ \ba \fk l_\lambda^i & =\text{the left kernel of the restriction of $B_\lambda$ to $\fk s_\lambda^{i-1}\times \fk s_{\lambda}^{i-1}$,} \\
\fk s_\lambda^i &=\text{the left kernel of the restriction of $B_\lambda$ to $\fk s_\lambda^{i-1} \times \fk l_\lambda^i$.}\ea\]

\begin{remark}
The degree  of the supercharacter $\chi_\lambda$ is 
\[ \chi_\lambda(1) = |G\lambda| =|\lambda G|= |\fkn| / |\fk l _\lambda^1|\qquad\text{and in addition}\qquad \langle \chi_\lambda,\chi_\lambda\rangle_G = |G\lambda \cap \lambda G| = |\fk s_\lambda^1| / |\fk l_\lambda^1|.\]  
\end{remark}We then have an ascending and descending chain of subspaces
\be\label{chain} 0  =\fk l_\lambda^0 \subset  \fk l_\lambda^1 \subset \fk l_\lambda^2 \subset \cdots\subset \fk s_\lambda^2 \subset \fk s_\lambda^1  \subset \fk s_\lambda^0 =  \fkn\ee with the following properties:
\begin{enumerate}
\item[(a)] Each $\fk s_\lambda^{i+1}$ is a subalgebra of $\fk s_\lambda^i$.
\item[(b)] Each $\fk l_\lambda^{i+1}$ is a right ideal of $\fk s_\lambda^i$ and a two-sided ideal of $\fk s_\lambda^{i+1}$.
\item[(c)] If $\fk s_\lambda^{d-1} = \fk s_\lambda^{d}$ for some $d\geq 1$ then $\fk l_\lambda^{d+i} = \fk l_\lambda^{d}$ and $\fk s_\lambda^{d+i} = \fk s_\lambda^d$ for all $i\geq 0$.
\end{enumerate}

Let $d\geq 1$ be an integer such that (c) holds$-$by dimensional considerations, some such $d$ exists$-$and define the subalgebras $\olfkl_\lambda,\olfks_\lambda\subset\fkn$ and algebra subgroups $\oll_\lambda,\ols_\lambda\subset G$ by \[  
\olfkl_\lambda = \fk l_\lambda^{d},
\qquad
 \olfks_\lambda = \fk s_\lambda^{d}\qquad\text{and}
 \qquad 
 \oll_\lambda = 1+\olfkl_\lambda,
 \qquad
 \ols_\lambda = 1+\olfks_\lambda.\] 
Observe that $\olfkl_\lambda$ and $\olfks_\lambda$ are just the terminal elements in the ascending and descending chains (\ref{chain}).

The function $\theta_\lambda: G \to \CC$ restricts to a linear character of $\oll_\lambda$ and we define the character $\xi_\lambda$ of $G$ by
\[\xi_\lambda = \Ind_{\oll_\lambda}^G(\theta_\lambda).\]
This character is a possibly reducible constituent  with  degree $|G|/|\oll_\lambda|$  of the supercharacter $\chi_\lambda$ of $G$.  Notably, if $\ols_\lambda =G$ then $\xi_\lambda = \chi_\lambda$.  To give these functions a formula, define
\[ \Xi_\lambda =\left\{ g\lambda sg^{-1} : g \in G,\ s \in \ols_\lambda\right\} \subset \fkn^*.\]  
One can show that  $\Xi_\mu = \Xi_\lambda$ if $\mu \in \Xi_\lambda$, and so the sets $\Xi_\lambda$ thus partition $\fkn^*$ into unions of coadjoint orbits.  The following result combines Corollary 3.1, Proposition 3.1, Theorem 3.3, Corollary 4.1, and Proposition 4.2 in \cite{supp0} and enumerates the properties of the characters $\xi_\lambda$ which will be of use in our applications.

\begin{theorem}\label{structural}
Let $\fkn$ be a finite-dimensional nilpotent $\FF_q$-algebra and write $G=1+\fkn$. If $\lambda,\mu \in \fkn^*$, then 
\begin{enumerate}
\item[(1)]  
$\displaystyle \xi_\lambda =\frac{|\ols_\lambda|}{\left|G\right|} \sum_{\nu \in \Xi_\lambda} \theta_\nu 
$ and $\displaystyle |\Xi_\lambda| = \frac{|G|^2}{| \oll_\lambda|| \ols_\lambda|}. $ 
\item[(2)] $\xi_\lambda = \xi_{\mu}$ if and only if $\mu \in \Xi_\lambda$, and if $\mu \notin \Xi_\lambda$ then $\xi_\lambda$ and $\xi_{\mu}$ share no irreducible constituents.   In particular, 
\[ \left\langle \xi_\lambda, \xi_{\mu} \right\rangle_G = \left\{\ba & |\ols_\lambda| / |\oll_\lambda|,&&\text{if }\mu \in \Xi_\lambda, \\&
0,&&\text{otherwise},\ea\right.\] and $\xi_\lambda$ is irreducible if and only if $\oll_\lambda =\ols_\lambda$, in which case $\xi_\lambda = \psi_\lambda = \logpsi_\lambda$.  

\item[(3)] The irreducible constituents of the characters $\{ \xi_\lambda : \lambda \in \fkn^*\}$ partition $\Irr(G)$, and the map \[ \barr{ccc} \Irr\(\ols_\lambda,\Ind_{\oll_\lambda}^{\ols_\lambda}(\theta_\lambda)\) & \to & \Irr(G,\xi_\lambda) \\ \psi & \mapsto & \Ind_{\ol S_\lambda}^{G}(\psi) \earr\] is a bijection.   Furthermore, if $\mu \in ( \olfks_\lambda )^*$ is given by restricting $\lambda$ to $\olfks_\lambda$, then the Kirillov function $\psi_\mu$ (respectively, $\logpsi_\mu$) is a character of $\ols_\lambda$ if and only if $\psi_\lambda$ (respectively, $\logpsi_\lambda$) is a character of $G$.  

\item[(4)] If $(\olfks_\lambda)^p \subset \olfkl_\lambda \cap \ker \lambda$  where $p>0$ is the characteristic of $\FF_q$, then $\logpsi_\lambda \in \Irr(G)$.

\end{enumerate}
\end{theorem}

\def\Gal{\mathrm{Gal}}

\begin{remark}
 Let $p>0$ be the characteristic of $\FF_q$ and suppose $r$ is a power of $p$ for which  $\QQ(\zeta_r)$ is a splitting field for $G$. 
 If $\alpha \in \Gal\(\QQ(\zeta_r) / \QQ(\zeta_p)\)$, 
then clearly $\alpha\circ \xi_\lambda = \xi_\lambda$.  Therefore, if $\chi $ is an irreducible constituent of $\xi_\lambda$ then $\alpha \circ \chi$ is as well.   In light of part (3) of the preceding statement, it follows that 
\[\label{rmk} \alpha \circ \psi = \psi\quad\text{if and only if}\quad \alpha \circ \Ind_{\ols_\lambda}^G(\psi) =  \Ind_{\ols_\lambda}^G(\psi),\qquad\text{for }\psi \in  \Irr\(\ols_\lambda,\Ind_{\oll_\lambda}^{\ols_\lambda}(\theta_\lambda)\) .\]
Now, for each $1\leq i \leq \log_p(r)$ there exists $\alpha \in \Gal\(\QQ(\zeta_r) / \QQ(\zeta_p)\)$ whose fixed field is $\QQ(\zeta_{p^i})$; namely, one can take $\alpha$ to be the unique automorphism with $\alpha(\zeta_r) = \zeta_r^{1+p^i}$.  Consequently, if $\psi$ takes values outside the field $\QQ(\zeta_{p^i})$, so that $\alpha \circ\psi \neq \psi$, then the same is true of the irreducible constituent $\Ind_{\ols_\lambda}^G(\psi)$ of $\xi_\lambda$.
\end{remark}

Here is an easy example of how one can use the constructions in this section to directly decompose a supercharacter.

\begin{example}
Consider the supercharacter $\chi_\lambda$  of $\UT_6(q)$ indexed by 
\[\lambda = e_{13}^* + e_{24}^* + e_{35}^* + e_{36}^* \in \fkt_6(q)^*.\]
It is an instructive exercise to compute 
\[ \ba 
\fk l_\lambda^1 &= \{ X \in \fkt_6(q) : X_{12} = X_{23} = X_{34} = X_{45}  =0\}, \\
\fk l^2_\lambda & = \{ X \in \fkt_6(q) : X_{12} = X_{23} = X_{45} = 0\}, \\
\fk l_\lambda^3 &= \fk s_\lambda^2\ea
\qquad
\ba
\fk s_\lambda^1 &= \{ X \in \fkt_6(q) : X_{45} = 0\}, \\
\fk s^2_\lambda & = \{ X \in \fkt_6(q): X_{23} = X_{45} = 0\}, \\
\fk s^3_\lambda&= \fk s^2_\lambda,
\ea\] so $\chi_\lambda$ has degree $q^4$ and   $\xi_\lambda$ is irreducible with degree $q^2$ since $\oll_\lambda = \ols_\lambda$.

The elements in the two-sided $\UT_6(q)$-orbit of $\lambda$ are those of the form $\mu = \lambda + \sum_{i} a_i e_{i,i+1}^*$, so $\mu(XY) = \lambda(XY)$ for all $X,Y \in \fkt_6(q)$.  This means by definition that $\oll_\lambda = \oll_\mu = \ols_\lambda = \ols_\mu$ and hence that each $\xi_\mu$ is irreducible with degree $q^2$.  As $\chi_\lambda$ is a linear combination of such characters $\xi_\mu$,  it follows that every irreducible constituent of $\chi_\lambda$ has degree $q^2$.  The irreducible constituents of a supercharacter which have the same degree also have the same multiplicity, since a constant times $\chi_\lambda$ is the character of a two-sided ideal in the group algebra of $\UT_n(q)$.  The constant in general is $|\fkn|/|\fk s_\lambda^1|$ and is in this case $q$, so since $\chi_\lambda(1) = q^4$ it follows that $\chi_\lambda$ decomposes as a sum of $q$ distinct irreducible characters of degree $q^2$, each appearing with multiplicity $q$.
\end{example}

\subsection{Inflation from quotients by algebra subgroups}\label{infl-sec}

\def\q{\mathfrak{q}}

We will make use in the next sections of the following result concerning the effect of inflation on the functions $\psi_\lambda$, $\logpsi_\lambda$, $\chi_\lambda$, $\xi_\lambda$.  
Suppose $\fkn$ is a nilpotent $\FF_q$-algebra with a two-sided ideal $\h$.  Let $\q=\fkn /\h$ be the quotient algebra, and write $G = 1+\fkn$ and $Q = 1+\q$.  Also, let $\wt \pi : \fkn \to \q$ be the quotient map, and define $\pi :G\to Q$ by $\pi(1+X) = 1+\wt\pi(X)$; both of these are surjective homomorphisms, of algebras and groups, respectively.
The following is proved as Observation 4.1 in \cite{supp0}.

\begin{observation} \label{inflation}  If $\lambda \in \fkn^*$ has $\ker \lambda \supset \h$, then there exists a unique $\mu \in \q^*$ with $\lambda = \mu\circ \wt\pi$, and 
\[\psi_\lambda = \psi_{\mu} \circ \pi,
\qquad
\logpsi_\lambda = \logpsi_{\mu} \circ \pi,
\qquad 
\chi_\lambda = \chi_{\mu } \circ\pi,
\qquad\text{and}
\qquad 
\xi_\lambda = \xi_{\mu} \circ \pi.\] Furthermore, $\oll_\lambda = \pi^{-1}(\oll_{\mu})$ and $\ols_\lambda = \pi^{-1}(\ols_\mu) $.
\end{observation}

\begin{remark} Since $\chi \mapsto \chi\circ \pi$ defines an injection $\Irr(Q)\to \Irr(G)$, by  \cite[Lemma 2.22]{Isaacs} for example, if $\psi_\mu$ or $\logpsi_\mu$ are characters in this setup then the same is true of $\psi_\lambda$ or $\logpsi_\lambda$, respectively.
\end{remark}

We will appeal to this observation most often in the special case when $\fkn$ has a vector space decomposition 
\be\label{above} \fkn = \fk a \oplus \h, \quad\text{where }\left\{\ba &\text{$\fk a$ is a subalgebra,} \\ &\text{$\h$ is a two-sided ideal.}\ea\right.\ee  Write $G = 1+\fkn$,  $A = 1+\fka$, and $H = 1+\h$.  In this situation, we may identify $Q = 1+\fkn /\h$ with the algebra subgroup $A$; Observation \ref{inflation} then holds where $\wt \pi :\fkn \to \fka$ and $\pi :G\to A$ are the projection maps 
\[ \wt \pi(a+h) = a\qquad\text{and}\qquad \pi(1+a+h)=1+a,\qquad\text{for }a\in \fka,\ h\in \h.\] 
Also, if $\lambda \in \fkn^*$ then the unique $\mu \in \fka^*$ with $\lambda = \mu\circ \wt \pi$ is given by the restriction $\mu = \lambda\downarrow \fka$.




\def\id{\mathrm{id}}


\section{Applications to the unitriangular group $\UT_n(q)$}\label{appl}

In this section we construct characters of the unitriangular group which take values in $\QQ(\zeta_{p^i})$ for any $i\geq 0$.  Approaching this goal, we first establish some general results concerning pattern groups, an accessible family of pattern groups, and then work to identify characters with large-field values in a specific algebra group, which will turn out to be a quotient of $\olfks_\lambda$ for a certain $\lambda \in \fkt_n(q)^*$.  Following this preliminary work, we identify our characters of interest as irreducible constituents of  one of the characters $\xi_\lambda$ of $\UT_n(q)$.

\subsection{Pattern groups}\label{pattern}

A \emph{pattern algebra} is any subalgebra of $\fkt_n(q)$ spanned over $\FF_q$ by a set of elementary matrices $e_{ij}$, and a \emph{pattern group} is an algebra group corresponding to a pattern algebra.  Pattern groups provide the most accessible examples of algebra groups, and much can be said about their supercharacters and class functions; see, for example, \cite{DT, I, MT, M2, TV}.

Given a subset of positions in an upper triangular matrix $\cP \subset \{ (i,j) : 1\leq i< j \leq n\}$, define 
\be\label{form} \fkt_{n,\cP}(q) = \{ X \in \fkt_n(q) : X_{ij} = 0 \text{ if }(i,j) \notin \cP\}\qquad\text{and}\qquad \UT_{n,\cP}(q) = 1+\fkt_{n,\cP}(q).\ee  It is not difficult to show that $\fkt_{n,\cP}(q)$ is a subalgebra if and only if $\cP$ is \emph{closed}, by which we mean that $(i,j),(j,k) \in \cP$ implies $(i,k) \in \cP$.  Every  pattern algebra and pattern group is thus of the form (\ref{form}) for a closed set of positions $\cP$.  

Such closed sets $\cP$ are naturally in bijection with  partial orderings of $[n]$ which are subordinate to the standard linear ordering $1<2<\dots<n$.  Specifically, $\cP$ corresponds to the ordering $\prec$ defined by $i\prec j$ if and only if $(i,j) \in\cP$; the closed condition on $\cP$ corresponds to the the transitivity condition on $\prec$.   The group $\UT_n(q)$ is the pattern group corresponding to the set of positions $\cP = \{ (i,j) : 1\leq i<j\leq n\}$ and the standard  linear ordering of  $[n]$.

Recall that a matrix is monomial if it has exactly one nonzero entry in each row and column.  Following \cite{Sangroniz}, we define a matrix to be \emph{quasi-monomial} if it has at most one nonzero entry in each row and column.  If $\fkn = \fkt_{n,\cP}(q)$ is a pattern algebra, then there is a natural isomorphism $\fkn \cong \fkn^*$ given by associating $X \in \fkn$ to the map $Y \mapsto \tr(X^TY) = \sum_{(i,j) \in \cP} X_{ij} Y_{ij}$ in $\fkn^*$.  We say that $\lambda \in \fkn^*$ is \emph{quasi-monomial} if $\lambda$ corresponds to a quasi-monomial matrix in $\fkn$ under this isomorphism.  Equivalently, if we define 
\[ \lambda_{ij} = \left\{\ba & \lambda(e_{ij}),&&\text{if }(i,j) \in \cP, \\ & 0,&&\text{if }(i,j)\notin\cP,\ea\right.\]
then $\lambda$ is quasi-monomial if $\lambda_{ij}\neq 0$ for at most one position $(i,j) \in \cP$ in each row and column.

The following easy lemma will be of great use in the calcuations we undertake in Section \ref{constructions}.

\def\cL{\leftexp{\perp}{\hspace{-0.5mm}\mathcal{L}}}
\def\cS{\leftexp{\perp}{\hspace{-0.5mm}\mathcal{S}}}

\begin{lemma}\label{monomial}
Suppose $\fkn = \fkt_{n,\cP}(q)$ is a pattern algebra and $\lambda \in \fkn^*$ is quasi-monomial.  If we define
\[ \ba \cL_\lambda &=\left \{ (i,j) \in \cP :\exists k\in [n]\text{ with }\lambda_{ik}\neq 0\text{ and }(j,k) \in \cP\right \}, \\
\cS_\lambda &= \left \{ (i,j) \in\cP : \exists k\in [n]\text{ with }\lambda_{ik}\neq 0\text{ and }(j,k) \in\cP\text{ and } (j,k) \notin  \cL_\lambda\right\}, \ea\] 
then $\lambda G = \lambda + \FF_q\spanning\left\{ e_{ij}^* : (i,j) \in \cL_\lambda\right\}$ and
for all $\mu \in \lambda G$ we have
\[  \fk l_\mu^1 = \left\{ X \in \fkt_{n,\cP}(q) : X_{ij}  =0\text{ if }(i,j) \in \cL_\lambda\right\}\quad\text{and}\quad
 \fk s_\mu^1 = \left\{ X \in \fkt_{n,\cP}(q) : X_{ij}  =0\text{ if }(i,j) \in \cS_\lambda\right\}.
\]  
\end{lemma}

\begin{remark}
Observe that if $\fkn = \fkt_n(q)$, so that $\cP = \{ (i,j) : i,j \in [n],\ i<j\}$, then $\cL_\lambda$ consists of all upper triangular positions strictly to the left of positions $(i,j)$ with $\lambda_{ij}\neq 0$, and \be\label{monomial-note} \cL_\lambda \setminus \cS_\lambda= \{ (i,j): \exists k,l \text{ such that } 1\leq i<j<k<l \leq n\text{ and }\lambda_{ik},\lambda_{jl}\neq 0\}.\ee
\end{remark} 

\begin{proof}
We have $\fk l^1_\mu = \fk l^1_\lambda$ and $\fk s_\mu^1 = \fk s_\lambda^1$ by \cite[Lemma 3.1]{supp0}.  Let $Y \in \fkt_{n,\cP}(q)$ and $(i,j) \in \cP$, so that $\lambda(e_{ij}Y)=\sum_{k \in [n]} \lambda_{ik} Y_{jk}$.  Since $\lambda$ is quasi-monomial, this is nonzero only if there exists $k \in [n]$ such that $\lambda_{ik} \neq 0$ and $Y_{jk} \neq 0$, in which case $(i,j) \in \cL_\lambda$.  It follows that 
\be\label{monomial-eq} \fk l_\lambda^1 \supset\left \{ X \in \fkt_{n,\cP}(q) : X_{ij}  =0\text{ if }(i,j) \in \cL_\lambda\right\}.\ee  On the other hand, if $X \in \FF_q\spanning \{ e_{ij} : (i,j) \in\cP\setminus \cL_\lambda\}$ then either $X= 0$ or $X_{ij} \neq 0$ for some $(i,j) \in \cP$ for which there exists $k \in [n]$ with $\lambda_{ij} \neq 0$ and $(j,k) \in \cP$.  In this case $\lambda(XY)=\lambda_{ik}Y_{ij} \neq 0$ for $Y= e_{jk} \in \fkt_{n,\cP}(q)$ since $\lambda$ is quasi-monomial and $XY$ has nonzero entries only in one column.  It follows that  (\ref{monomial-eq}) must be an equality.  The proof of our characterization of $\fk s_\mu^1 = \fk s_\lambda^1$ proceeds 
by an almost identical argument, and our description of  $\lambda G$ is immediate from \cite[Lemma 4.2(c)]{DI}.
%
\end{proof}

Our next result is a slight generalization of Theorem 7.2 in \cite{AndreAdjoint}.  It is noteworthy mostly for the almost trivial proof we can give using the preceding lemma.

\begin{theorem} If $\fkn = \fkt_{n,\cP}(q)$ is a pattern algebra and $\lambda \in \fkn^*$ is quasi-monomial then $\xi_\lambda = \psi_\lambda$; i.e., the Kirillov function $\psi_\lambda$ is a well-induced, irreducible character.
\end{theorem}

\begin{proof}
We note that if $\fkn =\fkt_{n,\cP}$ is a pattern algebra and $\lambda \in \fkn^*$ is quasi-monomial, then $\fk l_\lambda^1$, $\fk s_\lambda^1$ are also pattern algebras and $\lambda$ restricts to a quasi-monomial map $\fk s_\lambda^1\to \FF_q$.  Noting the chain (\ref{chain}), it suffices therefore to show that if $\lambda \in \fkn^*$ is quasi-monomial, then either $\fk l_\lambda^1 = \fk s_\lambda^1 $ or $\fk s_\lambda^1 \subsetneq \fkn$. This will force $\olfkl_\lambda = \olfks_\lambda$ by dimensional considerations, and thus $\xi_\lambda = \psi_\lambda$ by Theorem \ref{structural}.

To this end, suppose $\fk s_\lambda^1 = \fkn$ so that $\cS_\lambda = \varnothing$.  If $(i,j) \in \cL_\lambda$ so that there exists $k\in [n]$ with $\lambda_{ik}\neq0$ and $(j,k) \in \cP$, then $(j,k) \in \cL_\lambda$ since otherwise $(i,j) \in \cS_\lambda$.  Choosing $(i,j) \in \cL_\lambda$ with $j$ maximal and applying this argument thus gives a contradiction.  Therefore $\cL_\lambda=\varnothing$ so $\fk l_\lambda^1 = \fk s_\lambda^1$. 
\end{proof}

Before exiting this section, we discuss the following important fact due originally to Andr\'e \cite{Andre1} and Yan \cite{Yan}: the quasi-monomial maps $\lambda \in \fkt_n(q)^*$  index the distinct supercharacters of $\UT_n(q)$; i.e., the map 
\[\barr{ccc}   \bigl\{{ \text{Quasi-mononomial maps } \lambda \in \fkt_n(q)^*}\bigr \} &\to&\bigl\{{\text{Supercharacters of $\UT_n(q)$}}\bigr\} \\
\lambda&\mapsto& \chi_\lambda\earr\] is a bijection.
Furthermore, each quasi-monomial $\lambda$ naturally corresponds to a set partition of $[n]$.  This lends an interesting combinatorial interpretation to  many of the representation theoretic properties of the supercharacters $\chi_\lambda$.

In more detail, we recall that a \emph{set partition} of $[n]$ is a set $\Lambda = \{\Lambda_1,\dots,\Lambda_k\}$ of disjoint nonempty sets $\Lambda_i$ whose union is $[n]$. We call the sets $\Lambda_i$ the \emph{parts} of $\Lambda$ and write $\Lambda \vdash[n]$ to indicate that $[n]$ is the union of the parts of $\Lambda$.  We define the \emph{(unlabeled) shape} of a quasi-monomial $\lambda \in \fkt_n(q)^*$ as the finest set partition of $[n]$ in which $i,j$ belong to the same part whenever $\lambda_{ij} \neq 0$.  
Alternatively, the shape of $\lambda$ is the set partition whose parts are the vertex sets of the weakly connected components of the (weighted, directed) graph whose adjacency matrix is $\(\lambda_{ij}\)$.
For example, if $a,b,c \in \FF_q^\times$ then
\[\lambda= ae_{1,3}^* +b e_{2,4}^*  +ce_{3,5}^*  \in \fkt_6(q)^* \qquad\text{has shape}\qquad \{\{1,3,5\},\{2,4\}, \{6\}\}\vdash[6].\]  
The shape of a supercharacter $\chi$ of $\UT_n(q)$ is by definition the shape of the unique quasi-monomial $\lambda \in \fkt_n(q)^*$ with $\chi = \chi_\lambda$.
We introduce this terminology largely so that we can succinctly refer to the supercharacters of $\UT_n(q)$ which house our exotic irreducible characters as constituents.  The shape of a supercharacter encapsulates a good deal of less than obvious information about its irreducible constituents, however, a theme which we explore in greater detail in \cite{supp2},

\subsection{Complex characters of algebra groups}\label{cmplx-chars}

\def\mA{\mathrm{A}_n(q)}

Our strategy to construct exotic characters of $\UT_n(q)$ is to construct  them instead for the smaller algebra group $\ols_\lambda$ for some $\lambda\in \fkt_n(q)^*$.  To accomplish this, we must of course have some examples of algebra groups whose characters have values in large cyclotomic fields.  We provide a sort of quintessential construction here.

Fix a positive integer $n>1$ and  define  $\fka_n(q)$ as the nilpotent $\FF_q$-algebra  \be\label{fka-1} \fka_n(q) = \left\{ X \in \fkt_n(q) : X_{i+1,j+1} =X_{i,j}\text{ for }1\leq i<j < n\right\},\ee so that $\fka_n(q)$ consists of all $n\times n$-matrices over $\FF_q$ of the form
\[X = \(\barr{cccccc}
 0 & a_2 & a_3 & \cdots & a_n  \\
   & 0& a_2 & \ddots & \vdots  \\
 &     & 0 &  \ddots & a_3  \\
    &&  & \ddots & a_2 \\
    & & & & 0 
     \earr\),\qquad\text{with $a_i \in \FF_q$ and zeros below the diagonal}.\]     Let $\mA=1+\fka_n(q)$ be the corresponding algebra group.  
    If $X$ is any such matrix with $a_2\neq 0$ then the elements $X,X^2,\dots,X^{n-1}$ form a basis for $\fka_n(q)$  over $\FF_q$, and it follows that $\fka_n(q)$ is commutative and $\mA$ is abelian.
      
Let $\kappa \in \fka_n(q)^*$ be the linear map defined by 
\be \label{kappa} \kappa(X) =X_{1,n},\qquad\text{for $X \in \fka_n(q)$.}\ee  It is not difficult to see directly from the definitions in Section \ref{xi} that 
 $ \olfkl_\kappa = \FF_q\spanning\{e_{1,n}\} $ and 
 $\olfks_\kappa = \fk a_n(q)$.    Thus $\ols_\kappa =  \mA$ so $\xi_\kappa$ coincides with the supercharacter $\chi_\kappa$, and 
  we have 
  \be\label{formula} \chi_\kappa(g) = \Ind_{\oll_\kappa}^{\mA}(\theta_\kappa)(g) = \left\{\barr{ll} q^{n-2}\theta(g_{1,n}),&\text{if }g \in L_\kappa, \\ 0,&\text{otherwise},\earr\right.\qquad\text{for }g\in \mA.\ee
Since $\langle \chi_\kappa,\chi_\kappa\rangle_{\mA} = \chi_\kappa(1) = q^{n-2}$, the supercharacter $\chi_\kappa$ decomposes as a sum of $q^{n-2}$ distinct linear characters.  We can say more about these constituents:


\begin{proposition}\label{cmplx-constits}  Fix an integer $n>1$ and define $\kappa \in \fka_n(q)^*$ by (\ref{kappa}) as above. Let $p>0$ be the characteristic of $\FF_q$ and suppose $p^i$ is the largest power of $p$ less than $n$.  The following then hold:
\begin{enumerate}

\item[(1)] The values of the irreducible constituents of $\chi_\kappa$ are contained in the cyclotomic field $\QQ(\zeta_{p^{i+1}})$, and some irreducible constituent of $\chi_\kappa$ takes as values every $p^{i+1}$th root of unity.

\item[(2)] The Kirillov function $\psi_\kappa$ is a character of $\mA$ if and only if $n= 2$.

\item[(3)] The exponential Kirillov function $\logpsi_\kappa$ is a character of $\mA$ if and only if $n\leq p$.

\end{enumerate}
 \end{proposition}
 
 \begin{remark}  
The normalization of $\kappa \in \fka_n(q)^*$ is not important, as all holds still if one replaces $\kappa$ with $t\cdot \kappa$ for any nonzero $t \in \FF_q$; this just corresponds to a different choice of $\theta : \FF_q^+ \to \CC^\times$.  Also, regarding (3) we note that $\logpsi_{\nu}$ is a character  for every $\nu\in \fka_n(q)^*$ if $p\leq n$ by \cite[Corollary 3]{Sangroniz}.
 \end{remark}
 
 \begin{proof}
Write $\fka = \fka_n(q)$ and $G = \mA$, and let $X \in \fka$ be the matrix with $X_{1,2}= X_{2,3}=\dots=X_{n-1,n}=1$ and all other entries zero.  One checks that $X^i=0$ if and only if $i\geq n$, and hence that $Y^n=0$ for all $Y \in \fka$ since powers of $X$ span $\fka$. 

Let $r$  denote the largest power of $p$ less than $n$.  Then  $pr\geq n$ so $y^{pr}=1+(y-1)^{pr}=1$ for all $y \in G$ 
and as $G$ is an abelian $p$-group this implies the first half of (1). 
Observe that $x = 1+X \in G$ has order $pr$ since 
$x^{p^i} = 1+X^{p^i} \neq 1$ for $1\leq i \leq \log_p (r)$.  Let $H = \langle x \rangle$, 
and choose $\vartheta \in \Irr(H)$ such that $\vartheta(x)$ is a primitive $pr^{\mathrm{th}}$ root of unity with $\vartheta(x)^r = \theta(1)$.  Working from the definitions, it is straightforward to show that $\left\langle \chi_\kappa,\Ind_H^{G}(\vartheta) \right\rangle_{G}>0$ so
%
 at least one irreducible constituent of $\Ind_H^G(\vartheta) $ appears in $\chi_\kappa$.  
Since $G$ is abelian we have $\Ind_H^G(\vartheta)(x) = \frac{|G|}{|H|} \vartheta(x)$, and it follows that every irreducible constituent $\psi$ of $\Ind_H^G(\vartheta)$ has $\psi(x) = \vartheta(x)$, completing the proof of (1).

A function $G \to \CC$ is a character if and only if it defines a homomorphism $G\to \CC^\times$.  Since $\fka$ is commutative we have $\kappa^G = \{\kappa\}$ so $\psi_\kappa  = \theta_\kappa$.  This is a homomorphism if and only if $n= 2$, as one can check by considering its values at powers of $x=1+X$.  
If $n\leq p$ then $\logpsi_\kappa$ is a character by \cite[Corollary 3]{Sangroniz}.  Suppose $n>p$ and let $Z = X^{n-1}$.  Since $Z$ annihilates $\fka$, we have $\exp(Y+Z) = \exp(Y)+Z$ for all $Y \in \fka$.  If $n=p+1$  then for $x = \exp(X)$ and $z=1+Z = \exp(Z)$ we have $x^{n-1} =z$, but $\logpsi_\kappa(x)^{n-1} =\logpsi_\kappa(x)=1\neq \logpsi_\kappa(z) = \theta(1)$; thus $\logpsi_\kappa$ is not a homomorphism.
%
If $n\geq p+2$, one checks that if $a = \exp(X^{n-p})$ and $b = \exp(X)$ then 
\[\logpsi_\kappa(a) = \logpsi_\kappa(b) = 1\qquad\text{but}\qquad \logpsi_\kappa(ab) = \psi_\kappa(1+X^{n-p} +X-Z) = \theta(-1)\neq 1,\]
so
 $\logpsi_\kappa$ is again not a homomorphism, proving (3).
%
\end{proof}

\subsection{Exotic characters of $\UT_n(q)$}\label{constructions}

\def\block{\square}
\def\lt{\leftexp{\mathrm{L}}{\hspace{-0.8mm}\triangle}}
\def\ut{\leftexp{\mathrm{U}}{\hspace{-0.8mm}\triangle}}
\def\A{\cA}
\def\B{\cB}
\def\C{\cC}
\def\D{\cD}
\def\Z{\cZ}
\def\J{\cJ}

The goal of this section is to prove the following theorem promised in the introduction.

\begin{theorem}\label{main}
Let $p>0$ be the characteristic of $\FF_q$ and let $r=p^e$ for any integer $e>0$.  If $n > 6r$, then  $\UT_n(q)$ has an irreducible character of degree $q^{5r^2-2r}$ whose set of values is contained in $\QQ(\zeta_{pr})$ but not $\QQ(\zeta_r)$.  Such a character in fact occurs as an irreducible constituent of 
each supercharacter of $\UT_n(q)$ whose shape is the set partition of $[n]$ whose parts are the $n-4r-1$ sets 
\[ \ba \{ 1,2r+1,3r+1,4r+1,6r+1\}& \\
 \{ i,2r+i,3r+i,5r+i \}&\text{ for }1< i \leq r\\
 \{ i, 4r+1+i\} &\text{ for }r+1\leq i \leq 2r \\
\{ i\}&\text{ for }6r+1<i\leq n.
 \ea\]

\end{theorem}

 \definecolor{hellgrau}{gray}{0.5}
 
\def\Gtmp{G}
 \def\fkntmp{\fkn}
 
 \begin{remark}
 Setting $r=2$ and $n=13$ in this statement proves Conjecture 4.1 in \cite{IK05}.
 Figure \ref{fig0} below illustrates the set partition in the theorem when $r=2,4,8,16$  and $n=6r+1$.

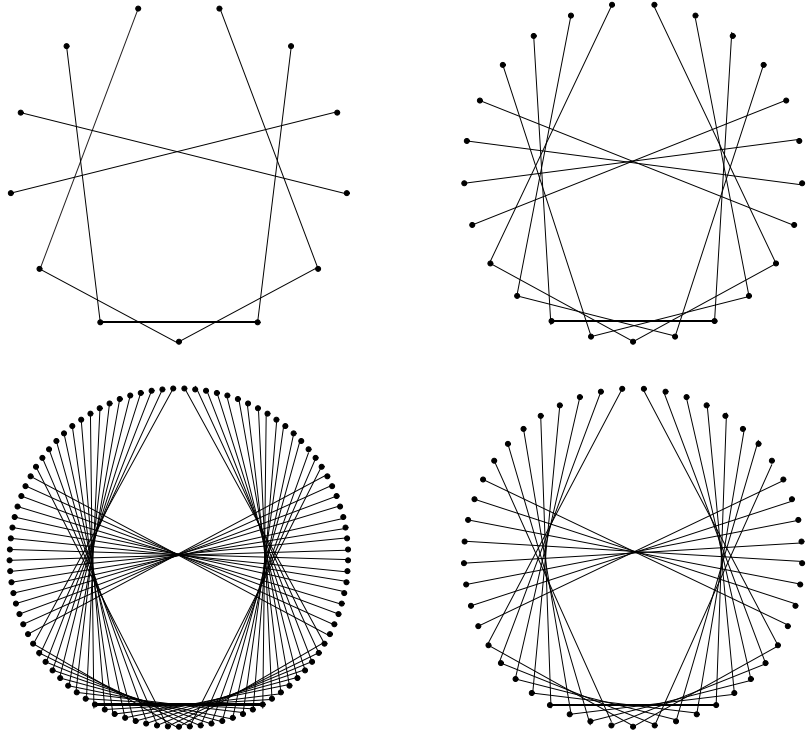
\begin{figure}[h]
\[\barr{c} 
\setlength{\unitlength}{4.5cm}
\begin{picture}(1.00, 1.00)
\put(0.3803,0.9855){\line(-29,-77){0.2918}}
\put(0.3803,0.9855){\circle*{0.015}}
\put(0.1684,0.8743){\line(10,-82){0.0992}}
\put(0.1684,0.8743){\circle*{0.015}}
\put(0.0325,0.6773){\line(96,-24){0.9639}}
\put(0.0325,0.6773){\circle*{0.015}}
\put(0.0036,0.4397){\line(96,24){0.9639}}
\put(0.0036,0.4397){\circle*{0.015}}
\put(0.0885,0.2160){\line(41,-22){0.4115}}
\put(0.0885,0.2160){\circle*{0.015}}
\put(0.2676,0.0573){\line(46,0){0.4647}}
\put(0.2676,0.0573){\circle*{0.015}}
\put(0.5000,0.0000){\line(41,22){0.4115}}
\put(0.5000,0.0000){\circle*{0.015}}
\put(0.7324,0.0573){\line(10,82){0.0992}}
\put(0.7324,0.0573){\circle*{0.015}}
\put(0.9115,0.2160){\line(-29,77){0.2918}}
\put(0.9115,0.2160){\circle*{0.015}}
\put(0.9964,0.4397){\circle*{0.015}}
\put(0.9675,0.6773){\circle*{0.015}}
\put(0.8316,0.8743){\circle*{0.015}}
\put(0.6197,0.9855){\circle*{0.015}}
\end{picture}
\qquad\qquad
\setlength{\unitlength}{4.5cm}
\begin{picture}(1.00, 1.00)
\put(0.4373,0.9961){\line(-36,-76){0.3595}}
\put(0.4373,0.9961){\circle*{0.015}}
\put(0.3159,0.9649){\line(-16,-83){0.1582}}
\put(0.3159,0.9649){\circle*{0.015}}
\put(0.2061,0.9045){\line(5,-84){0.051}}
\put(0.2061,0.9045){\circle*{0.015}}
\put(0.1147,0.8187){\line(26,-80){0.2609}}
\put(0.1147,0.8187){\circle*{0.015}}
\put(0.0476,0.7129){\line(93,-37){0.9279}}
\put(0.0476,0.7129){\circle*{0.015}}
\put(0.0089,0.5937){\line(99,-13){0.9902}}
\put(0.0089,0.5937){\circle*{0.015}}
\put(0.0010,0.4686){\line(99,13){0.9902}}
\put(0.0010,0.4686){\circle*{0.015}}
\put(0.0245,0.3455){\line(93,37){0.9279}}
\put(0.0245,0.3455){\circle*{0.015}}
\put(0.0778,0.2321){\line(42,-23){0.4222}}
\put(0.0778,0.2321){\circle*{0.015}}
\put(0.1577,0.1355){\line(47,-12){0.4666}}
\put(0.1577,0.1355){\circle*{0.015}}
\put(0.2591,0.0618){\line(48,0){0.4818}}
\put(0.2591,0.0618){\circle*{0.015}}
\put(0.3757,0.0157){\line(47,12){0.4666}}
\put(0.3757,0.0157){\circle*{0.015}}
\put(0.5000,0.0000){\line(42,23){0.4222}}
\put(0.5000,0.0000){\circle*{0.015}}
\put(0.6243,0.0157){\line(26,80){0.2609}}
\put(0.6243,0.0157){\circle*{0.015}}
\put(0.7409,0.0618){\line(5,84){0.051}}
\put(0.7409,0.0618){\circle*{0.015}}
\put(0.8423,0.1355){\line(-16,83){0.1582}}
\put(0.8423,0.1355){\circle*{0.015}}
\put(0.9222,0.2321){\line(-36,76){0.3595}}
\put(0.9222,0.2321){\circle*{0.015}}
\put(0.9755,0.3455){\circle*{0.015}}
\put(0.9990,0.4686){\circle*{0.015}}
\put(0.9911,0.5937){\circle*{0.015}}
\put(0.9524,0.7129){\circle*{0.015}}
\put(0.8853,0.8187){\circle*{0.015}}
\put(0.7939,0.9045){\circle*{0.015}}
\put(0.6841,0.9649){\circle*{0.015}}
\put(0.5627,0.9961){\circle*{0.015}}
\end{picture}
\\ \\ 
\setlength{\unitlength}{4.5cm}
\begin{picture}(1.00, 1.00)
\put(0.483809,0.999738){\line(-41,-75){0.414098}}
\put(0.4838,0.9997){\circle*{0.015}}
\put(0.451495,0.997642){\line(-36,-78){0.364397}}
\put(0.4515,0.9976){\circle*{0.015}}
\put(0.419385,0.993458){\line(-31,-80){0.313168}}
\put(0.4194,0.9935){\circle*{0.015}}
\put(0.387612,0.987205){\line(-26,-82){0.260626}}
\put(0.3876,0.9872){\circle*{0.015}}
\put(0.356311,0.978909){\line(-21,-84){0.206990}}
\put(0.3563,0.9789){\circle*{0.015}}
\put(0.325613,0.968603){\line(-15,-85){0.15}}
\put(0.3256,0.9686){\circle*{0.015}}
\put(0.295646,0.956333){\line(-10,-86){0.099}}
\put(0.2956,0.9563){\circle*{0.015}}
\put(0.266536,0.942148){\line(-4,-86){0.040}}
\put(0.2665,0.9421){\circle*{0.015}}
\put(0.238406,0.928108){\line(1,-86){0.0104}}
\put(0.2384,0.9261){\circle*{0.015}}
\put(0.211372,0.908282){\line(7,-86){0.069604}}
\put(0.2114,0.9083){\circle*{0.015}}
\put(0.185550,0.888743){\line(12,-85){0.12}}
\put(0.1855,0.8887){\circle*{0.015}}
\put(0.161046,0.867573){\line(18,-84){0.179833}}
\put(0.1610,0.8676){\circle*{0.015}}
\put(0.137964,0.844862){\line(23,-83){0.23}}
\put(0.1380,0.8449){\circle*{0.015}}
\put(0.116400,0.820704){\line(29,-81){0.287048}}
\put(0.1164,0.8207){\circle*{0.015}}
\put(0.096445,0.795201){\line(34,-79){0.338960}}
\put(0.0964,0.7952){\circle*{0.015}}
\put(0.078183,0.768460){\line(39,-77){0.389452}}
\put(0.0782,0.7685){\circle*{0.015}}
\put(0.061691,0.740593){\line(88,-47){0.884180}}
\put(0.0617,0.7406){\circle*{0.015}}
\put(0.047036,0.711717){\line(91,-41){0.912545}}
\put(0.0470,0.7117){\circle*{0.015}}
\put(0.034282,0.681952){\line(94,-35){0.937084}}
\put(0.0343,0.6820){\circle*{0.015}}
\put(0.023481,0.651425){\line(96,-29){0.957692}}
\put(0.0235,0.6514){\circle*{0.015}}
\put(0.014679,0.620262){\line(97,-22){0.974283}}
\put(0.0147,0.6203){\circle*{0.015}}
\put(0.007912,0.588595){\line(99,-16){0.986787}}
\put(0.0079,0.5886){\circle*{0.015}}
\put(0.003209,0.556557){\line(100,-10){0.995153}}
\put(0.0032,0.5566){\circle*{0.015}}
\put(0.000590,0.524281){\line(100,-3){0.999345}}
\put(0.0006,0.5243){\circle*{0.015}}
\put(0.000066,0.491903){\line(100,3){0.999345}}
\put(0.0001,0.4919){\circle*{0.015}}
\put(0.001638,0.459560){\line(100,10){0.995153}}
\put(0.0016,0.4596){\circle*{0.015}}
\put(0.005301,0.427386){\line(99,16){0.986787}}
\put(0.0053,0.4274){\circle*{0.015}}
\put(0.011039,0.395516){\line(97,22){0.974283}}
\put(0.0110,0.3955){\circle*{0.015}}
\put(0.018827,0.364085){\line(96,29){0.957692}}
\put(0.0188,0.3641){\circle*{0.015}}
\put(0.028634,0.333224){\line(94,35){0.937084}}
\put(0.0286,0.3332){\circle*{0.015}}
\put(0.040418,0.303062){\line(91,41){0.912545}}
\put(0.0404,0.3031){\circle*{0.015}}
\put(0.054130,0.273726){\line(88,47){0.884180}}
\put(0.0541,0.2737){\circle*{0.015}}
\put(0.069711,0.245340){\line(43,-25){0.430289}}
\put(0.0697,0.2453){\circle*{0.015}}
\put(0.087098,0.218021){\line(45,-22){0.445267}}
\put(0.0871,0.2180){\circle*{0.015}}
\put(0.106216,0.191886){\line(46,-19){0.458378}}
\put(0.1062,0.1919){\circle*{0.015}}
\put(0.126986,0.167042){\line(47,-16){0.469566}}
\put(0.1270,0.1670){\circle*{0.015}}
\put(0.149321,0.143596){\line(48,-13){0.478785}}
\put(0.1493,0.1436){\circle*{0.015}}
\put(0.173126,0.121644){\line(49,-10){0.485995}}
\put(0.1731,0.1216){\circle*{0.015}}
\put(0.198303,0.101279){\line(49,-6){0.491167}}
\put(0.1983,0.1013){\circle*{0.015}}
\put(0.224745,0.082586){\line(49,-3){0.494279}}
\put(0.2247,0.0826){\circle*{0.015}}
\put(0.252341,0.065644){\line(50,0){0.495318}}
\put(0.2523,0.0656){\circle*{0.015}}
\put(0.280976,0.050524){\line(49,3){0.494279}}
\put(0.2810,0.0505){\circle*{0.015}}
\put(0.310530,0.037289){\line(49,6){0.491167}}
\put(0.3105,0.0373){\circle*{0.015}}
\put(0.340879,0.025995){\line(49,10){0.485995}}
\put(0.3409,0.0260){\circle*{0.015}}
\put(0.371894,0.016690){\line(48,13){0.478785}}
\put(0.3719,0.0167){\circle*{0.015}}
\put(0.403448,0.009411){\line(47,16){0.469566}}
\put(0.4034,0.0094){\circle*{0.015}}
\put(0.435406,0.004190){\line(46,19){0.458378}}
\put(0.4354,0.0042){\circle*{0.015}}
\put(0.467635,0.001049){\line(45,22){0.445267}}
\put(0.4676,0.0010){\circle*{0.015}}
\put(0.500000,0.000000){\line(43,25){0.430289}}
\put(0.5000,0.0000){\circle*{0.015}}
\put(0.532365,0.001049){\line(39,77){0.389452}}
\put(0.5324,0.0010){\circle*{0.015}}
\put(0.564594,0.004190){\line(34,79){0.338960}}
\put(0.5646,0.0042){\circle*{0.015}}
\put(0.596552,0.009411){\line(29,81){0.287048}}
\put(0.5966,0.0094){\circle*{0.015}}
\put(0.628106,0.016690){\line(23,83){0.23}}
\put(0.6281,0.0167){\circle*{0.015}}
\put(0.659121,0.025995){\line(18,84){0.179833}}
\put(0.6591,0.0260){\circle*{0.015}}
\put(0.689470,0.037289){\line(12,85){0.12}}
\put(0.6895,0.0373){\circle*{0.015}}
\put(0.719024,0.050524){\line(7,86){0.069604}}
\put(0.7190,0.0505){\circle*{0.015}}
\put(0.757659,0.928644){\line(-1,-86){0.0102}}
\put(0.7477,0.0656){\circle*{0.015}}
\put(0.775255,0.082586){\line(-4,86){0.040}}
\put(0.7753,0.0826){\circle*{0.015}}
\put(0.801697,0.101279){\line(-10,86){0.10}}
\put(0.8017,0.1013){\circle*{0.015}}
\put(0.826874,0.121644){\line(-15,85){0.15}}
\put(0.8269,0.1216){\circle*{0.015}}
\put(0.850679,0.143596){\line(-21,84){0.206990}}
\put(0.8507,0.1436){\circle*{0.015}}
\put(0.873014,0.167042){\line(-26,82){0.260626}}
\put(0.8730,0.1670){\circle*{0.015}}
\put(0.893784,0.191886){\line(-31,80){0.313168}}
\put(0.8938,0.1919){\circle*{0.015}}
\put(0.912902,0.218021){\line(-36,78){0.364397}}
\put(0.9129,0.2180){\circle*{0.015}}
\put(0.930289,0.245340){\line(-41,75){0.414098}}
\put(0.9303,0.2453){\circle*{0.015}}
\put(0.9459,0.2737){\circle*{0.015}}
\put(0.9596,0.3031){\circle*{0.015}}
\put(0.9714,0.3332){\circle*{0.015}}
\put(0.9812,0.3641){\circle*{0.015}}
\put(0.9890,0.3955){\circle*{0.015}}
\put(0.9947,0.4274){\circle*{0.015}}
\put(0.9984,0.4596){\circle*{0.015}}
\put(0.9999,0.4919){\circle*{0.015}}
\put(0.9994,0.5243){\circle*{0.015}}
\put(0.9968,0.5566){\circle*{0.015}}
\put(0.9921,0.5886){\circle*{0.015}}
\put(0.9853,0.6203){\circle*{0.015}}
\put(0.9765,0.6514){\circle*{0.015}}
\put(0.9657,0.6820){\circle*{0.015}}
\put(0.9530,0.7117){\circle*{0.015}}
\put(0.9383,0.7406){\circle*{0.015}}
\put(0.9218,0.7685){\circle*{0.015}}
\put(0.9036,0.7952){\circle*{0.015}}
\put(0.8836,0.8207){\circle*{0.015}}
\put(0.8620,0.8449){\circle*{0.015}}
\put(0.8390,0.8676){\circle*{0.015}}
\put(0.8145,0.8887){\circle*{0.015}}
\put(0.7886,0.9083){\circle*{0.015}}
\put(0.7616,0.9261){\circle*{0.015}}
\put(0.7335,0.9421){\circle*{0.015}}
\put(0.7044,0.9563){\circle*{0.015}}
\put(0.6744,0.9686){\circle*{0.015}}
\put(0.6437,0.9789){\circle*{0.015}}
\put(0.6124,0.9872){\circle*{0.015}}
\put(0.5806,0.9935){\circle*{0.015}}
\put(0.5485,0.9976){\circle*{0.015}}
\put(0.5162,0.9997){\circle*{0.015}}
\end{picture}
\qquad\qquad
\setlength{\unitlength}{4.5cm}
\begin{picture}(1.00, 1.00)
\put(0.467965,0.998973){\line(-40,-76){0.395536}}
\put(0.4680,0.9990){\circle*{0.015}}
\put(0.404421,0.990780){\line(-30,-80){0.30}}
\put(0.4044,0.9908){\circle*{0.015}}
\put(0.342446,0.974528){\line(-19,-83){0.190287}}
\put(0.3424,0.9745){\circle*{0.015}}
\put(0.283058,0.950484){\line(-8,-85){0.081}}
\put(0.2831,0.9505){\circle*{0.015}}
\put(0.227233,0.919044){\line(3,-85){0.0305}}
\put(0.2272,0.9190){\circle*{0.015}}
\put(0.175886,0.880723){\line(14,-84){0.14}}
\put(0.1759,0.8807){\circle*{0.015}}
\put(0.129861,0.836150){\line(24,-82){0.243312}}
\put(0.1299,0.8362){\circle*{0.015}}
\put(0.089914,0.786058){\line(35,-78){0.346148}}
\put(0.0899,0.7861){\circle*{0.015}}
\put(0.056700,0.731269){\line(90,-43){0.900506}}
\put(0.0567,0.7313){\circle*{0.015}}
\put(0.030766,0.672683){\line(95,-31){0.948568}}
\put(0.0308,0.6727){\circle*{0.015}}
\put(0.012536,0.611260){\line(98,-19){0.981055}}
\put(0.0125,0.6113){\circle*{0.015}}
\put(0.002310,0.548012){\line(100,-6){0.997433}}
\put(0.0023,0.5480){\circle*{0.015}}
\put(0.000257,0.483974){\line(100,6){0.997433}}
\put(0.0003,0.4840){\circle*{0.015}}
\put(0.006409,0.420200){\line(98,19){0.981055}}
\put(0.0064,0.4202){\circle*{0.015}}
\put(0.020666,0.357736){\line(95,31){0.948568}}
\put(0.0207,0.3577){\circle*{0.015}}
\put(0.042794,0.297608){\line(90,43){0.900506}}
\put(0.0428,0.2976){\circle*{0.015}}
\put(0.072429,0.240804){\line(43,-24){0.427571}}
\put(0.0724,0.2408){\circle*{0.015}}
\put(0.109084,0.188255){\line(45,-18){0.454854}}
\put(0.1091,0.1883){\circle*{0.015}}
\put(0.152159,0.140825){\line(47,-12){0.474669}}
\put(0.1522,0.1408){\circle*{0.015}}
\put(0.200945,0.099293){\line(49,-6){0.486689}}
\put(0.2009,0.0993){\circle*{0.015}}
\put(0.254641,0.064341){\line(49,0){0.490718}}
\put(0.2546,0.0643){\circle*{0.015}}
\put(0.312366,0.036542){\line(49,6){0.486689}}
\put(0.3124,0.0365){\circle*{0.015}}
\put(0.373173,0.016353){\line(47,12){0.474669}}
\put(0.3732,0.0164){\circle*{0.015}}
\put(0.436061,0.004105){\line(45,18){0.454854}}
\put(0.4361,0.0041){\circle*{0.015}}
\put(0.500000,0.000000){\line(43,24){0.427571}}
\put(0.5000,0.0000){\circle*{0.015}}
\put(0.563939,0.004105){\line(35,78){0.346148}}
\put(0.5639,0.0041){\circle*{0.015}}
\put(0.626827,0.016353){\line(24,82){0.243312}}
\put(0.6268,0.0164){\circle*{0.015}}
\put(0.687634,0.036542){\line(14,84){0.14}}
\put(0.6876,0.0365){\circle*{0.015}}
\put(0.745359,0.064341){\line(3,85){0.0305}}
\put(0.7454,0.0643){\circle*{0.015}}
\put(0.799055,0.099293){\line(-8,85){0.081}}
\put(0.7991,0.0993){\circle*{0.015}}
\put(0.847841,0.140825){\line(-19,83){0.190287}}
\put(0.8478,0.1408){\circle*{0.015}}
\put(0.890916,0.188255){\line(-30,80){0.30}}
\put(0.8909,0.1883){\circle*{0.015}}
\put(0.927571,0.240804){\line(-40,76){0.395536}}
\put(0.9276,0.2408){\circle*{0.015}}
\put(0.9572,0.2976){\circle*{0.015}}
\put(0.9793,0.3577){\circle*{0.015}}
\put(0.9936,0.4202){\circle*{0.015}}
\put(0.9997,0.4840){\circle*{0.015}}
\put(0.9977,0.5480){\circle*{0.015}}
\put(0.9875,0.6113){\circle*{0.015}}
\put(0.9692,0.6727){\circle*{0.015}}
\put(0.9433,0.7313){\circle*{0.015}}
\put(0.9101,0.7861){\circle*{0.015}}
\put(0.8701,0.8362){\circle*{0.015}}
\put(0.8241,0.8807){\circle*{0.015}}
\put(0.7728,0.9190){\circle*{0.015}}
\put(0.7169,0.9505){\circle*{0.015}}
\put(0.6576,0.9745){\circle*{0.015}}
\put(0.5956,0.9908){\circle*{0.015}}
\put(0.5320,0.9990){\circle*{0.015}}
\end{picture}
\earr
\]
\caption{
{Set partitions of $[6\cdot 2^i+1]$ for $i=1,2,3,4$ indexing
supercharacters of $\UT_{6\cdot 2^i+1}(2)$ with irreducible constituents that have values in $\QQ(\zeta_{2^{i+1}})\setminus \QQ(\zeta_{2^i})$.}}
\label{fig0}
\end{figure}
\end{remark}


 For the rest of this section, we fix an integer $r>1$ (not necessarily a prime power) and let $\Gtmp = \UT_{6r+1}(q)$ and $\fkntmp = \fkt_{6r+1}(q)$.  We shall prove the theorem by the following steps:
  \begin{enumerate}
 \item[1.] First, we will compute $\olfks_\lambda$ for a certain map $\lambda \in \fkntmp^*$.
 
 \item[2.] We will then identify a quotient of $\ols_\lambda$ isomorphic to the group $\mathrm{A}_{r+1}(q)$ defined in Section \ref{cmplx-chars}.
  
 \item[3.] We will then demonstrate that the supercharacter of $\ols_\lambda$ indexed by the restriction $\mu= \lambda \downarrow \olfks_\lambda$ is equal to the product of a linear supercharacter and a character obtained by inflating the supercharacter  of $\mathrm{A}_{r+1}(q)$ indexed by the map $\kappa \in \fka_{r+1}(q)^*$ defined by (\ref{kappa}).
 
 \end{enumerate}
   This will show that we can view the characters in the set $\Irr(\ols_\lambda,\chi_\mu)$ as products of a linear supercharacter with the $q^{r-1}$ linear constituents of $\chi_\kappa$ whose values are discussed in Proposition \ref{cmplx-constits}.  These characters become the irreducible constituents of $\xi_\lambda$ on induction to $\Gtmp$ by Theorem \ref{structural}, and the remarks following that corollary together with Proposition \ref{cmplx-constits} imply that some of the induced characters have values which lie in $\QQ(\zeta_{pr})$ but not $\QQ(\zeta_r)$.  

To begin this program, let us define the map $\lambda \in \fkntmp^*$ of interest.  If we view $\lambda$ as the matrix whose $(i,j)$th entry is $\lambda_{ij}\overset{\mathrm{def}}=\lambda(e_{ij})$, then $\lambda$ informally corresponds to the picture in Figure \ref{fig1}.  This diagram is meant to illustrate a $(6r+1)\times (6r+1)$ upper triangular matrix; the dark diagonal lines mark the positions $(i,j)$ where $\lambda_{ij} \neq 0$.  To achieve our result these nonzero entries can be arbitrary, but to make our computations neater we will set them all to be $\pm 1$.
 
 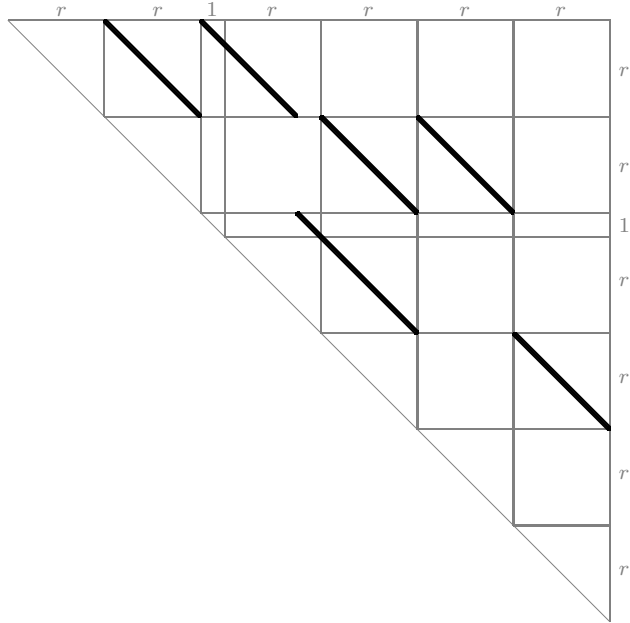
\begin{figure}[h]
 \[
 \barr{c} \\
  {
\setlength{\unitlength}{8cm}
\begin{picture}(1.01, 1.01)
\color{hellgrau}

\put(0.08,1.015){$_r$}
\put(0.24,1.015){$_r$}
\put(0.33,1.015){$_1$}
\put(0.43,1.015){$_r$}
\put(0.59,1.015){$_r$}
\put(0.75,1.015){$_r$}
\put(0.91,1.015){$_r$}

\put(1.015,0.085){$_r$}
\put(1.015,0.245){$_r$}
\put(1.015,0.405){$_r$}
\put(1.015,0.565){$_r$}
\put(1.015,0.655){$_1$}
\put(1.015,0.755){$_r$}
\put(1.015,0.915){$_r$}

\thinlines
  \put(1, 0){\line(-1, 1){1}}
  \put(0, 1){\line(1, 0){1}}
    \put(1, 0){\line(0, 1){1}}

    \put(0.16,0.84){\line(0,1){0.16}}
        \put(0.16,0.84){\line(1,0){0.84}}
        
    \put(0.32,0.68){\line(0,1){0.32}}
        \put(0.32,0.68){\line(1,0){0.68}}
       \put(0.36,0.64){\line(1,0){0.64}}
       \put(0.36,0.64){\line(0,1){0.36}}
              
        
    \put(0.52,0.48){\line(0,1){0.52}}
    \put(0.52,0.48){\line(1,0){0.48}}
    
    \put(0.68,0.32){\line(0,1){0.68}} 
     \put(0.68,0.32){\line(1,0){0.32}} 
        
    \put(0.84,0.16){\line(0,1){0.84}} 
        \put(0.84,0.16){\line(1,0){0.16}} 

\color{black}
\thicklines

              \put(0.68,.84){\line(1,-1){0.16}}
	 \put(0.68,.84){\line(1,0){0.004}}
	  \put(0.68,.84){\line(0,-1){0.004}}
            \put(0.684,.84){\line(1,-1){0.156}}
              \put(0.68,.836){\line(1,-1){0.156}}
               \put(0.84,.68){\line(-1,0){0.004}}
	  \put(0.84,.68){\line(0,1){0.004}}
              
          \put(0.52,.84){\line(1,-1){0.16}}
	 \put(0.52,.84){\line(1,0){0.004}}
	  \put(0.52,.84){\line(0,-1){0.004}}
            \put(0.524,.84){\line(1,-1){0.156}}
              \put(0.52,.836){\line(1,-1){0.156}}
               \put(0.68,.68){\line(-1,0){0.004}}
	  \put(0.68,.68){\line(0,1){0.004}}

          \put(0.32,1){\line(1,-1){0.16}}
	 \put(0.32,1){\line(1,0){0.004}}
	  \put(0.32,1){\line(0,-1){0.004}}
            \put(0.324,1){\line(1,-1){0.156}}
              \put(0.32,.996){\line(1,-1){0.156}}
               \put(0.48,.84){\line(-1,0){0.004}}
	  \put(0.48,.84){\line(0,1){0.004}}	
	  
          \put(0.16,1){\line(1,-1){0.16}}
	 \put(0.16,1){\line(1,0){0.004}}
	  \put(0.16,1){\line(0,-1){0.004}}
            \put(0.164,1){\line(1,-1){0.156}}
              \put(0.16,.996){\line(1,-1){0.156}}
               \put(0.32,.84){\line(-1,0){0.004}}
	  \put(0.32,.84){\line(0,1){0.004}}
	  
	 \put(0.48,.68){\line(1,-1){0.20}}
	 \put(0.48,.68){\line(1,0){0.004}}
	  \put(0.48,.68){\line(0,-1){0.004}}
            \put(0.484,.68){\line(1,-1){0.196}}
              \put(0.48,.676){\line(1,-1){0.196}}
               \put(0.68,.48){\line(-1,0){0.004}}
	  \put(0.68,.48){\line(0,1){0.004}}	  
              
              \put(0.84,.48){\line(1,-1){0.16}}
	 \put(0.84,.48){\line(1,0){0.004}}
	  \put(0.84,.48){\line(0,-1){0.004}}
            \put(0.844,.48){\line(1,-1){0.156}}
              \put(0.84,.476){\line(1,-1){0.156}}
               \put(1,.32){\line(-1,0){0.004}}
	  \put(1,.32){\line(0,1){0.004}}

\end{picture}
}
 \earr
\]
\caption{A picture representing the $\FF_q$-linear map $\lambda \in \fkt_{6r+1}(q)^*$}
\label{fig1}
\end{figure}
To give a more precise definition, we briefly use the notation 
\[ \sigma(n;i,j) \overset{\mathrm{def}}= \sum_{k=1}^n e_{i+k,j+k}^* \in \fkntmp^*.\]
For the duration of this section, we define $\lambda \in \fkntmp^*$ by
 $ \lambda =\lambda' - \lambda''$ where 
 \[\ba \lambda ' &=  \sigma(r;0,2r) + \sigma(r;r,4r+1) + \sigma(r;3r+1,5r+1)+ \sigma(r+1;2r,3r) ,
\\
\lambda'' &=
  \sigma(r;0,r) + \sigma(r;r,3r+1)
 .
\ea \] Note that $\lambda'$ is quasi-monomial with shape given by the set partition in Theorem \ref{main}, and that $\lambda \in \lambda'\Gtmp$ by Lemma \ref{monomial}.  Comparing this definition to Figure \ref{fig1}, one observes that $\lambda$ has six ``pieces'' corresponding to subsets of positions on various diagonals; exactly one such subset has $r+1$ positions and the rest have $r$ positions.  It is not difficult to check that this definition may be recast as the piecewise formula
 \be\label{lambda} \lambda_{jk}=\lambda(e_{jk}) = \left\{\ba 
 &-1,&&\text{if }j=k-r \in [1,r] \\
 &1,&&\text{if }j=k-2r \in [1,r]\\
   &-1,&&\text{if }j=k-2r-1\in [r+1,2r]\\
    &1,&&\text{if }j=k-3r-1 \in [r+1,2r] \\
  &1,&&\text{if }j=k-r\in[2r+1,3r+1]\\
    &1,&&\text{if }j=k-2r \in [3r+2,4r+1]\\
    &0,&&\text{otherwise}.
 \ea\right.\qquad
 \ee
This last identity is what we will use to actually compute $\lambda(X)$ for $X \in \fkntmp$; our arguments depend much more intuitively, however, on the visual representation of $\lambda$ given in Figure \ref{fig1}.

To compute and describe the subalgebras $\olfkl_\lambda$ and $\olfks_\lambda$ we require several technical definitions referring to subsets of positions in an upper triangular matrix.  We begin this lexicon by defining $\J$ as the set of all such positions:
\[\J = \{ (i,j) : i,j \in [6r+1] : i<j\}.\]  Our task is now to define the eleven subsets $\A,\B,\C,\D,\Z_1,\dots,\Z_7\subset \J$ corresponding to regions in a $(6r+1)\times(6r+1)$ upper triangular matrix highlighted in Figure \ref{fig2} below.

\def\imap{\tau}
 \definecolor{hellgrau}{gray}{0.5}
 
 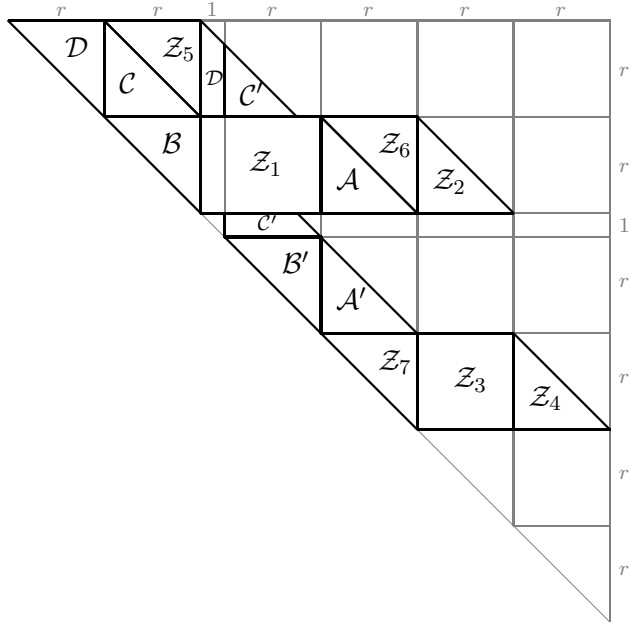
\begin{figure}[h]
 \[
 \barr{c}\\
{
\setlength{\unitlength}{8cm}
\begin{picture}(1.01, 1.01)

\color{hellgrau}

\put(0.08,1.015){$_r$}
\put(0.24,1.015){$_r$}
\put(0.33,1.015){$_1$}
\put(0.43,1.015){$_r$}
\put(0.59,1.015){$_r$}
\put(0.75,1.015){$_r$}
\put(0.91,1.015){$_r$}

\put(1.015,0.085){$_r$}
\put(1.015,0.245){$_r$}
\put(1.015,0.405){$_r$}
\put(1.015,0.565){$_r$}
\put(1.015,0.655){$_1$}
\put(1.015,0.755){$_r$}
\put(1.015,0.915){$_r$}

\thinlines
  \put(1, 0){\line(-1, 1){1}}
  \put(0, 1){\line(1, 0){1}}
    \put(1, 0){\line(0, 1){1}}

    \put(0.16,0.84){\line(0,1){0.16}}
        \put(0.16,0.84){\line(1,0){0.84}}
        
    \put(0.32,0.68){\line(0,1){0.32}}
        \put(0.32,0.68){\line(1,0){0.68}}
       \put(0.36,0.64){\line(1,0){0.64}}
       \put(0.36,0.64){\line(0,1){0.36}}
              
        
    \put(0.52,0.48){\line(0,1){0.52}}
    \put(0.52,0.48){\line(1,0){0.48}}
    
    \put(0.68,0.32){\line(0,1){0.68}} 
     \put(0.68,0.32){\line(1,0){0.32}} 
        
    \put(0.84,0.16){\line(0,1){0.84}} 
        \put(0.84,0.16){\line(1,0){0.16}} 

\color{black}
\thicklines


            \put(0.40, 0.75){$\Z_1$}        
    \put(0.32,.84){\line(1,0){0.20}}
    \put(0.32,.84){\line(0,-1){0.16}}
        \put(0.32,0.68){\line(1,0){0.20}}
            \put(0.52,.68){\line(0,1){0.16}}       
                    
   
             \put(0.74, 0.39){$\Z_3$} 
              \put(0.68,.48){\line(1,0){0.16}}
                  \put(0.68,.32){\line(0,1){0.16}}   
                   \put(0.68,.32){\line(1,0){0.16}}
                  \put(0.84,.32){\line(0,1){0.16}}

          \put(0.705,.72){$\Z_2$}
              \put(0.68,.84){\line(1,-1){0.16}}
            \put(0.68,.68){\line(1,0){0.16}}

          \put(0.865,.36){$\Z_4$}
              \put(0.84,.48){\line(0,-1){0.16}}
              \put(0.84,.48){\line(1,-1){0.16}}
            \put(0.84,.32){\line(1,0){0.16}}

          \put(0.615,.42){$\Z_7$}
              \put(0.68,.48){\line(0,-1){0.16}}
              \put(0.52,.48){\line(1,-1){0.16}}
            \put(0.52,.48){\line(1,0){0.16}}

          \put(0.615,.78){$\Z_6$}
              \put(0.68,.84){\line(0,-1){0.16}}
              \put(0.52,.84){\line(1,-1){0.16}}
            \put(0.52,.84){\line(1,0){0.16}}

          \put(0.255,.94){$\Z_5$}
              \put(0.32,1){\line(0,-1){0.16}}
              \put(0.16,1){\line(1,-1){0.16}}
            \put(0.16,1){\line(1,0){0.16}}

    \put(0.545,.72){$\A$}
              \put(0.52,.84){\line(1,-1){0.16}}
            \put(0.52,.68){\line(0,1){0.16}}
            \put(0.52,.68){\line(1,0){0.16}}

          \put(0.255,.78){$\B$}
              \put(0.32,.84){\line(0,-1){0.16}}
              \put(0.16,.84){\line(1,-1){0.16}}
            \put(0.16,.84){\line(1,0){0.16}}
            
\put(0.185,.88){$\C$}
              \put(0.16,1){\line(1,-1){0.16}}
            \put(0.16,.84){\line(0,1){0.16}}
            \put(0.16,0.84){\line(1,0){0.16}}

          \put(0.095,.94){$\D$}
              \put(0.16,1){\line(0,-1){0.16}}
              \put(0.0,1){\line(1,-1){0.16}}
            \put(0.0,1){\line(1,0){0.16}}

   \put(0.328,0.90){$_\D$}
            \put(0.36,.96){\line(-1,1){0.04}}
            \put(0.32,.96){\line(0,-1){0.12}}
             \put(0.36,.96){\line(0,-1){0.12}}
              \put(0.32,.84){\line(1,0){0.04}}


            \put(0.545,.52){$\A'$}
              \put(0.52,.64){\line(1,-1){0.16}}
            \put(0.52,.48){\line(1,0){0.16}}
     
          \put(0.455,.58){$\B'$}
              \put(0.52,.64){\line(0,-1){0.16}}
              \put(0.36,.64){\line(1,-1){0.16}}
            \put(0.36,.64){\line(1,0){0.16}}

\put(0.385,.86){$\C'$}
              \put(0.36,.96){\line(1,-1){0.12}}
            \put(0.36,.84){\line(0,1){0.12}}
            \put(0.36,0.84){\line(1,0){0.12}}
            \put(0.415,.66){$_{\C'}$}
   \put(0.36,0.64){\line(0,1){0.04}}
   \put(0.48,0.68){\line(1,-1){0.04}}
   \put(0.36,0.68){\line(1,0){0.12}}
      \put(0.36,0.64){\line(1,0){0.12}}            
\end{picture}
}
 \earr
\]
\caption{$\A,\B,\C,\D,\Z_1,\dots,\Z_7 \subset \J$ are sets of positions above the diagonal in a matrix}
\label{fig2}
\end{figure}
%
%
%
%
%
%
%
%
%
%
%
%
%
%
%
%
%
%
As indicated by our picture, these subsets for the most part correspond to blocks of adjacent positions which lie inside triangles or rectangles.  
This diagram is somewhat imprecise; among other deficiencies, it does not clearly indicate how the sets in question include positions on various diagonals.  However, this picture will serve as a valuable heuristic in what follows.
To state our definitions in more adequate detail, we adopt the following notation: let 
\[
 \ba \block(n;x,y) &= \{ (x+i,y+j) : i,j \in [n] \}, \\
\lt(n;x,y) &= \{ (x+j,y+i) : i,j \in [n],\ i<j \}, \\
\ut(n;x,y) &= \{ (x+i,y+j) : i,j \in [n],\ i<j\},\ea\qquad\text{for nonnegative integers $n,x,y$}.\] 
Thus $\block(n;x,y)$ is the $n$-by-$n$ square of positions containing $(x+1,y+1)$ and $(x+n,y+n)$, and $\lt(n;x,y)$ and $\ut(n;x,y) $ are the subsets of $\block(n;x,y)$ consisting of the positions strictly below the diagonal and strictly above the diagonal.  Using these notations, we define
%
\[ \ba 
&\ba 
\A&= \lt(r;r,3r+1), \\
\B&= \ut(r;r,r), \\
\C&=\lt(r;0,r), \\
\ea\qquad\ba
\A'&=\lt(r;2r+1,3r+1),\\
\B'&=\ut(r;2r+1,2r+1),\\
\C' &= \lt(r-1;1,2r+1) \cup \{ (2r+1,2r+i): i \in [2,r]\},\ea
\\
\\[-10pt]
&\ba
\D&=\ut(r;0,0) \cup \{ (i,2r+1) : i \in [2,r]\},
\ea
\ea
\] and 
\[
\ba
&
 \ba \Z_1 & = \block(r;r,2r) \cup \{ (2r+i,3r+1) : i \in [r]\}, \\
 \Z_2 & =  \lt(r;r,4r+1),\\
  \Z_3 &= \block(r;3r+1,4r+1), \ea\qquad\ba
  \Z_5 &=\block(r;0,r) \setminus \C, \\
  \Z_6 & = \block(r;r,3r+1) \setminus \A, \\
   \Z_7 &= \ut(r;3r+1,3r+1).\ea \\
&\ba \Z_4 &= \lt(r;3r+1,5r+1), \\ 
 \ea\ea
 \] These formulas at first glance appear forbiddingly technical, but the definitions are easily interpreted with the aid of our picture above.  In particular, one observes that the sets are all disjoint and correspond to the regions in Figure \ref{fig2}.  Let 
 \[ \Z = \Z_1 \cup \Z_2 \cup \Z_3 \cup \Z_4\qquad\text{and}\qquad \Z' = \Z_5\cup \Z_6 \cup \Z_7.\]  It is apparent from Figures 1 and 2 that the sets $\cL_{\lambda'}$ and $\cS_{\lambda'}$ defined in Lemma \ref{monomial} 
  are given by
 \be\label{monomial-app}\ba \cL_{\lambda'} &= \A \cup \A' \cup \B\cup\B' \cup \C\cup\C' \cup\D\cup\Z \cup\Z', \\
 \cS_{\lambda'} &= \Z.\ea\ee  Thus we may immediately compute $\fk l_\lambda^1$ and $\fk s_\lambda^1$ from Lemma \ref{monomial}.

The subalgebras $\fk l_\lambda^i,\fk s_\lambda^i$ defined in Section \ref{xi} will consist of elements $X \in \fkntmp$ such that $X_\alpha =0$ for certain positions $\alpha \in \J$ and such that $X_{\alpha_1} = \dots = X_{\alpha_k}$ for certain positions $\alpha_1,\dots,\alpha_k \in \J$.  A succinct way of stating conditions of the second type is to define a map $\imap: \J \to \J$ and then stipulate that $X_\alpha = X_{\imap(\alpha)}$ for all $\alpha \in \J$.  This motivates our next and last definition:
let $\imap : \A\cup \B \cup \C \cup \D \to \A'\cup\B'\cup\C'\cup \D$ be the map given by
 \[\ba  \imap(i,j) &= (i+r+1,j),&&\qquad\text{for }(i,j) \in \A, \\
 \imap(i,j) &= (i+r+1,j+r+1),&&\qquad\text{for }(i,j) \in \B, \\ 
 \imap(i,j) &= \left\{\ba
 &
 (i+1,j+r+1),&&\quad\text{if }i<r, \\
&
 (2r+1,j+r+1),&&\quad\text{if }i=r, 
 \ea\right.
 &
 &
 \qquad\text{for }(i,j) \in \C, \\ 
 \imap(i,j) &= \left\{\ba
 &
 (i+1,j+1),&&\quad\text{if }j<r, \\
 &
 (i+1,2r+1),&&\quad\text{if }j=r, \\
 &
 (1,r+2-i),&&\quad\text{if }j=2r+1, 
 \ea\right.
 &
 &
 \qquad\text{for }(i,j) \in \D. 
 \ea\]  
 Comparing this formula with Figure \ref{fig2} makes things much more comprehensible.  We in particular note the following.
\begin{remarks}

\begin{enumerate}
\item[]
\item[(i)] Observe that $\imap$ is injective with $\imap(\A) =\A'$, $\imap(\B)=\B'$, $\imap(\C)=\C'$, and $\imap(\D) = \D$.  

\item[(ii)] Note further that $\imap$ is ``orientation-preserving'' on $\A\cup \B\cup\C$, in the sense that $\tau :\X \to \X'$ for $\X = \A,\B,\C$ is the unique bijection which preserves the relative locations of any two positions (so that if $(i,j) \in \A$ is to the left of $(k,\ell)\in\A$ then $\tau(i,j)\in\A'$ is to the left of $\tau(k,\ell)\in\A'$, for example).

%

\item[(iii)] If $X \in \fkntmp$ has $X_{\alpha} = X_{\imap(\alpha)}$ for all $\alpha \in \cD$ then $X_{i,2r+1} = X_{i-1,r} =X_{i-2,r-1}= \dots = X_{1,r-i+2}$ for all $i\in[2,r]$.  Thus, if $\fka \subset \fkn$ is the subspace
\be\label{fka} \fka = \{ X \in \fkn : X_{\alpha} = X_{\imap(\alpha)} \text{ if }\alpha \in \D\text{ and }X_\alpha =0\text{ if }\alpha \notin \D \} \oplus \FF_q\spanning\{ e_{1,2r+1}\}\ee then $\fka$ is a subalgebra naturally isomorphic to the algebra $\fka_{r+1}(q)$ defined in Section \ref{cmplx-chars}. The map $\fka \to \fka_{r+1}(q)$ defined by
$X \mapsto Y$ where 
\[  Y_{ij} =\left\{\barr{ll} X_{ij},&\text{if $j \leq r,$} \\
  X_{i,2r+1},&\text{if }j=r,\earr\right.\qquad\text{for $i,j \in [r+1]$}\] 
  gives an isomorphism.
 
 \end{enumerate}
 \end{remarks}

\def\S{\mathcal{S}}
\def\T{\mathcal{T}}

 With these definitions and remarks, we may now state the following technical lemma.  The proof of this is tedious but, with proper organization, is not as difficult as it might appear. Notably our proof does not require a computer, 
 and this makes the results in this section  apparently the first statements concerning exotic character values of $\UT_n(q)$ (see, for example, \cite{E,IK1,IK2,IK05,VeraLopez2004}) which can be derived entirely by hand.

 \begin{lemma}\label{technical}  Fix an integer $r>1$ and let $\fkn = \fkt_{6r+1}(q)$.  If $\lambda \in \fkntmp^*$ is defined by (\ref{lambda}), then
  \[\ba \fk l_\lambda^1 &= \{ X \in \fkn : 
 X_{\alpha} = 0\text{ if } \alpha  \in \A\cup\A'\cup \B\cup \B' \cup \C\cup\C' \cup \D \cup \Z \cup \Z' \}, \\
\fk l_\lambda^2 &= \{ X \in \fkn : X_\alpha = X_{\imap(\alpha)} \text{ if }\alpha \in \A \text{ and }X_{\alpha} = 0\text{ if } \alpha  \in \B\cup\B' \cup \C \cup\C'\cup \D \cup \Z \}, \\
\olfkl_\lambda=\fk l_\lambda^3 &= \{ X \in \fkn : X_\alpha = X_{\imap(\alpha)} \text{ if }\alpha \in \A\cup\B\cup\C  \text{ and }X_{\alpha} = 0\text{ if } \alpha  \in  \D \cup \Z \},
\\[-10pt]
\\
 \fk s_\lambda^1 &= \{ X \in \fkn : X_\alpha = 0\text{ if }\alpha \in \Z \}, \\
\fk s_\lambda^2 &= \{ X \in \fkn : X_{\alpha} = X_{\imap(\alpha)}\text{ if }\alpha \in \A\cup \B \text{ and } X_\alpha = 0\text{ if }\alpha \in \Z  \}, \\
\olfks_\lambda=\fk s_\lambda^3 &= \{ X \in \fkn : X_\alpha = X_{\imap(\alpha)} \text{ if }\alpha \in \A\cup \B\cup\C\cup\D  \text{ and } X_{\alpha} = 0\text{ if } \alpha  \in  \Z  \}. \ea\]
Also, $(\lambda -e_{1,2r+1}^*)(XY) =0$ for all $X,Y \in \olfks_\lambda$.
 
 \end{lemma}
 
It may be helpful to observe that one could also write $\olfkl_\lambda = \{ X \in \olfks_\lambda : X_\alpha =0\text{ for }\alpha \in \D\}$, and that for $r=2$ and $n=13$ we are claiming 
\[
    \olfks_\lambda ={\small \left\{ 
 \(\barr{ccccccccccccc}
 0 & d & * & * & * & * & * & * & * & * & * & * & *\\
   & 0 & c & * & d & * & * & * & * & * & * & * & *\\
   &   & 0 & b & 0 & 0 & 0 & * & * & * & * & * & *\\
  &  &  & 0 & 0 & 0 & 0 & a & * & 0 & * & * & *\\
   &  &  &  & 0 & c & * & * & * & * & * & * & *\\
  &  &  &  &  & 0 & b & * & * & * & * & * & *\\
  &  &  &  &  &  & 0 & a & * & * & * & * & *\\
  &  &  &  &  &  &  & 0 & * & 0 & 0 & * & *\\
  &  &  &  &  &  &  &  & 0 & 0 & 0 & 0 & *\\
  &  &  &  &  & &  & &  & 0 & *& * & *\\
  &  &  &  &  & &  & &  &  &0 & * & *\\
  &  &  &  &  & &  & &  &  & & 0 & *\\
    &  &  &  &  & &  & &  &  &  &  & 0     \earr\)\right\},}
    \]  where the parameters $a,b,c,d$ correspond to the regions of the same letter.

\begin{proof}  Call the right hand sets above $\fk l_i$ and $\fk s_i$  for $i\leq 3$ and define $\fk l_4 = \fk l_3$ and $\fk s_4= \fk s_3$.  Recalling the discussion in Section \ref{xi}, the lemma then becomes  the claim that $\fk l_i = \fk l^i_\lambda$ and $\fk s_i=\fk s^i_\lambda$ for $i\leq 4$.  
It is immediate from Lemma \ref{monomial} and the observation (\ref{monomial-app})  that $\fk l_1 = \fk l_\lambda^1$ and $\fk s_1 = \fk s_\lambda^1$.   
For the other cases, define the following subspaces in $\fkntmp = \fkt_n(q)$:
\[ \ba \fk l_2^c &= \FF_q\spanning \left\{ e_\alpha : \alpha \in \A\cup \B \cup  \B'  \cup \C \cup  \C'  \cup \D\right\},
\\
\fk l_3^c &=\FF_q\spanning \left\{ e_\alpha : \alpha \in \C \cup \D \right\}, 
\\
\fk l_4^c &= \left\{X \in \fkn : X_{\alpha} = X_{\imap(\alpha)}\text{ if }\alpha \in \D \text{ and }X_\alpha =0\text{ if }\alpha\notin\D\right\},
\\[-10pt]
\\ \fk s_2^c &= \FF_q\spanning \left\{ e_\alpha : \alpha \in \A \cup \B\right\}, 
\\
\fk s_3^c &=\FF_q\spanning\left\{ e_\alpha : \alpha \in \C \cup \ut(r;0,0) \right\},
\\
\fk s_4^c & = 0,
\\[-10pt]
\\ \fk l_2' &=\FF_q\spanning \left\{ e_\alpha+e_{\imap(\alpha)} : \alpha \in \A\right\} \oplus \FF_q\spanning\left\{ e_\alpha : \alpha \in \Z'\right \}, \\
\fk l_3' &= \FF_q\spanning\left\{ e_\alpha + e_{\imap(\alpha)} : \alpha \in \B \cup \C\right\}, \\
\fk l_4' &= 0.
\ea\] 
One checks that 
$ \fk s_{i-1} = \fk l_{i} \oplus \fk l_i^c = \fk s_i \oplus \fk s_i^c $ and $\fk l_i = \fk l_i' \oplus \fk l_{i-1}$ for $i \in \{2,3,4\}$.  To prove the lemma, it thus suffices to show that if $i \in \{2,3,4\}$ then 
\begin{enumerate}
\item[(a)] For each $X \in \fk l_i'$ we have $\lambda(XY) = 0$ for all $Y \in \fk s_{i-1}$.
\item[(b)] For each nonzero $X \in \fk l_i^c$ we have $\lambda(XY) \neq 0$ for some $Y \in \fk s_{i-1}$.
\item[(c)] For each $X \in \fk s_i$ we have $\lambda(XY) = 0$ for all $Y \in \fk l_i'$.
\item[(d)] For each nonzero $X \in \fk s_i^c$ we have $\lambda(XY) \neq 0$ for some $Y \in \fk l_i'$.
\end{enumerate}
In particular, (a) and (b) together imply that $\fk l_i$ is the left kernel of the bilinear form $B_\lambda : (X,Y)\mapsto \lambda(XY)$ restricted to $\fk s_{i-1} \times \fk s_{i-1}$, and (c) and (d) together imply that $\fk s_i$ is the left kernel of $B_\lambda$ restricted to $\fk s_{i-1}\times \fk l_i$, and these statements mean that $\fk l_i = \fk l_\lambda^i$ and $\fk s_i = \fk s_\lambda^i$ as required.  We have three cases ($i=2,3,4$) which we treat in turn; the first is by far the most technical, but all are handled by straightforward, elementary considerations.

\begin{case}[Suppose $i=2$]
\end{case}
 \begin{enumerate}
 \item[(a)] Let $(j,k) \in \J-\Z$ and set $Y = e_{jk} \in \fk s_1$.  Elements of this form span $\fk s_1$, so to show that (a) holds we need only prove that $\lambda(XY) =0$ for all elements $X$ in a basis for $\fk l_2'$.  For this, we have two cases: 
\begin{enumerate}
\item[i.]  Suppose $X = e_\alpha +e_{\imap(\alpha)} \in \fk l_2'$ for some  $\alpha  \in \A$.  Then $XY= 0$ unless $\alpha = (i,j)$ for some $1\leq i < j$, and in this case we have by (\ref{lambda}) that 
\be\label{A-case}\lambda(XY) =\lambda_{ik} + \lambda_{i+r+1,k} 
= \left\{\barr{rl} 1-0=1,&\text{if }k=i+3r+1, \\
-1+1=0,&\text{if }k=i+2r+1, \\
 0-0=0,&\text{otherwise}.\earr\right.\ee The case $k=i+3r+1$ does not occur, since then $(i,j) \in \A$ would imply that $j \in 3r+1+[r]$ and $k \in 4r+1+[r]$, whence $(j,k) \in \Z_3 \subset \Z$, a contradiction.

\item[ii.] Suppose $X= e_\alpha \in \fk l_2'$ for some $\alpha \in \Z'$.  It follows from (\ref{lambda}) that if $\lambda_{ik} \neq 0$ 
then either $(i,j) \notin \Z'$ or $(j,k) \in \Z$, and this observation suffices to show that $\lambda(XY)=0$.

\end{enumerate} 
The elements $X$ in (i) and (ii) span $\fk l_2'$, so (a) holds.
 
 \item[(b)]
Suppose $X \in \fk l_2^c$ is nonzero, so that $X_\alpha \neq 0$ for some $\alpha = (i,j) \in \A \cup \B \cup  \B'  \cup \C \cup  \C'  \cup \D$.  Let $Y = e_\beta$ where 
\be\label{beta} \beta = \left\{ \barr{ll} (j,i+2r+1),&\text{if }\alpha \in \A\cup \B, 
\\
\\[-5pt]
(j,i+r),&\text{if }\alpha \in  \B' \cup \C \cup\underbrace{ \{ (2r+1,2r+k) : k\in[2,r]\} }_{\subset  \C' } \cup \underbrace{\ut(r;0,0)}_{\subset \D}, 
\\
\\[-5pt]
(j, i+2r),&\text{if }\alpha \in\underbrace{ \lt(r-1;1,2r+1)}_{\subset  \C' } \cup\underbrace{ \{ (i,2r+1) : i \in [2,r]\}}_{\subset \D} . \\
\earr\right.\ee  One checks that $\beta \in \Z_7 \cup \A$ in the first case; $\beta \in  \C'  \cup \B \cup  \B'  \cup \C$ in the second case; and $\beta \in  \B'  \cup  \C' $ in the third case, respectively.  Hence $Y \in \fk s_1$.  Since $X$ has no nonzero entries in the positions $\Z_1 \cup  \A' $, it follows from (\ref{lambda}) that $\lambda(XY) = \pm X_\alpha \neq 0$, as required.

%

\item[(c)] To show that (c) holds, we have again two cases:
\begin{enumerate}
\item[i.] Suppose $Y = e_\alpha +e_{\imap(\alpha)} \in \fk l_2'$ for some $\alpha  \in \A$.  If $(i,j) \in \J$ then it follows from (\ref{lambda}) that
\[ \lambda(e_{ij}Y) = \left\{\barr{rl} -1,&\text{if }\alpha = (j,i+2r+1), \\
1,&\text{if }\alpha = (j-r-1,i+2r+1), \\
1,&\text{if }\alpha = (j-r-1,i+r), \\
0,&\text{otherwise}.\earr\right.\]  The first case occurs only if $(i,j) \in \B$;  the second case occurs only if $(i,j) \in \Z_1$; and the third case occurs only if $(i,j) \in  \B' $.  
Furthermore, noting the definition of $\imap$ on $\B$, one finds using this identity that if $\beta \in \B$ then $\lambda(e_\beta Y) =-1$ if and only if $\lambda(e_{\imap(\beta)}Y)=1$.  Since every $X \in \fk s_2$ has $X_\beta = 0$ if $\beta \in \Z$ and $X_\beta = X_{\imap(\beta)}$ if $\beta \in \B$, we have  $\lambda(XY) = 0$ for all $X  \in \fk s_2$.

\item[ii.] If  $Y = e_\alpha \in \fk l_2'$ for some $\alpha \in \Z_5  \cup \Z_6 $, then it is easy to see that $\lambda(XY)=0$ for all $X \in \fkn$.  If  $Y = e_\alpha \in \fk l_2'$ for some $\alpha \in \Z_7$ and $(i,j) \in \J$ then 
\[ \lambda(e_{ij}Y) = \left\{ \barr{rl} -1,&\text{if }\alpha= (j,i+2r+1), \\
1,&\text{if }\alpha = (j,i+r), \\
0,&\text{otherwise}.\earr\right.\]  The first case occurs only if $(i,j) \in \A$ and the second case occurs only if $(i,j) \in  \A' $.    Furthermore,  it follows from this identity that if $\beta \in \A$ then $\lambda(e_\beta Y) =-1$ if and only if $\lambda(e_{\imap(\beta)}Y) = 1$.  Since every $X \in \fk s_2$ has $X_\beta = X_{\imap(\beta)}$ for $\beta \in \A$, we  again have  $\lambda(XY) = 0$ for all $X  \in \fk s_2$.

\end{enumerate}
The elements $Y$ in (i) and (ii) span $\fk l_2'$, so this suffices to prove (c).

\item[(d)]
Suppose $X \in \fk s_2^c$ is nonzero, so that $X_\alpha \neq 0$ for some $\alpha =(i,j) \in \A \cup \B$.  Let $Y = e_\beta$ with $\beta =(j,i+2r+1)\in \J$.  If $\alpha \in \A$ then $\beta \in \Z_7$ and if $\alpha \in \B$ then $\beta \in \A$, so $Y \in \fk l_2'$.  
Since $X$ has nonzero entries only in $\A \cup \B$, it follows from (\ref{lambda}) that $\lambda(XY) = \pm X_\alpha \neq 0$, as required.
 
 \end{enumerate}

\begin{case}[Suppose $i=3$]
\end{case}

 \begin{enumerate}
 \item[(a)] It suffices to show that if $X = e_\alpha +e_{\imap(\alpha)} \in \fk l_3'$ for some $\alpha \in \A$, then $\lambda(XY) = 0$ for all $Y \in \fk s_2$.  Since $Y_{\beta} = 0$ for all $\beta \in \Z$ if $Y \in \fk s_2$, this follows directly from (\ref{A-case}), which shows that $\lambda(X e_{\beta}) = 0$ if $\beta \notin \Z_3\subset \Z$.
 
 \item[(b)]
Suppose $X \in \fk l_3^c$ is nonzero, so that $X_\alpha \neq 0$ for some $\alpha=(i,j) \in \C\cup \D$.    Define 
\[ Y = \left\{\barr{lll} e_\beta + e_{\imap(\beta)},&\text{ for }\beta = (j,i+r) \in \B,&\text{if }\alpha \in \D\text{ and }j\neq 2r+1, \\
e_\beta,&\text{ for }\beta = (j,i+r) \in \C,&\text{if }\alpha \in \C, \\
e_\beta,&\text{ for }\beta=(2r+1,2r+i) \in  \C' ,&\text{if }\alpha \in \D\text{ and }j = 2r+1.\earr\right.\]In each case we have $Y \in \fk s_2$ and one checks that $\lambda(XY) = \pm X_\alpha \neq 0$, as required.

\item[(c)] To see that (c) holds we have two cases:
\begin{enumerate}
\item[i.] Suppose $Y = e_\beta +e_{ \imap(\beta)} \in \fk l_3'$ for some $\beta=(j,k) \in \B$.   Let $\gamma = (k-r,j)$ and observe that $\beta \in \B$ implies $\gamma \in \C$.  If $X \in \fkn$ then 
\[ \lambda(XY)  = \left\{\barr{rl} -X_{k-r,j} + X_{k-r+1,j+r+1},&\text{if }k\in r + [r-1], \\
-X_{r,j} + X_{2r+1,j+r+1},&\text{if }k=2r.\earr\right.\]  It follows from the definition of $\imap$ that $ \lambda(XY) = X_{\imap(\gamma)}-X_{\gamma}$.  If $X \in \fk s_3'$ then 
$X_{\gamma} =X_{\imap(\gamma)}$, which implies that $\lambda(XY) = 0$.

\item[ii.] Alternatively suppose $Y = e_\gamma + \imap(e_\gamma) \in \fk l_3'$ for some $\gamma=(j,k) \in \C$.   Let $\delta = (k-r,j)$ and observe that $\gamma \in \C$ implies $\delta \in \D$.  If $X \in \fkntmp$ then 
\[ \lambda(XY)  = \left\{\barr{rl} -X_{k-r,j} + X_{k-r+1,j+1},&\text{if }j\in[ r-1] , \\
-X_{k-r,r} + X_{k-r+1,2r+1},&\text{if }j=r.\earr\right.\]  It again follows from the definition of $\imap$ that $ \lambda(XY) = X_{\imap(\delta)}-X_{\delta}$.  If $X \in \fk s_3$ then by definition $X_{\delta} =X_{\imap(\delta)}$, which implies that $\lambda(XY) = 0$.
\end{enumerate}
The elements $Y$ in (i) and (ii) span $\fk l_3'$, so (c) holds.

\item[(d)]
Suppose $X \in \fk s_3^c$ is nonzero, so that $X_\alpha \neq 0$ for some $\alpha=(i,j) \in \C\cup  \ut^r_{0,0}$.  We may assume without loss of generality that $X$ has no nonzero entries to the right of column $j$.  Let $Y=e_\beta + e_{\imap(\beta)}$ where $\beta = (j,i+r)$.  If $\alpha \in \C$ then $\beta \in \B$ and if $\alpha \in \ut_{0,0}^r$, then $\beta \in \C$.  Thus in either case $Y \in \fk l_3'$ and $\imap(\beta)$ occurs in a row strictly below the row of $\beta$.  Since we assume that $X$ has no nonzero entries in any column to the right of $\alpha$, it follows that $XY = Xe_\beta$ so $\lambda(XY) =-X_\alpha \neq 0$ by (\ref{lambda}), as required.

\end{enumerate}

\begin{case}[Suppose $i=4$]
\end{case}

 \begin{enumerate}
 \item[(a)]  Since $\fk l_4'=0$, (a) holds trivially.
 
 \item[(b)] If $X \in \fk l_4^c$ is nonzero then $X_{1i}\neq 0$ for some $i\in [2,r]$.  If we take $Y \in \fk s_3$ to be the element 
 \[ Y =  e_{i,2r+1} + e_{i-1,r} + e_{i-2,r-1} + \dots + e_{1,r-i+2}\] then $\lambda(XY) = -X_{1i} \neq 0.$
 
 

\item[(c)]  Since $\fk l_4'=0$, (c) also holds trivially.

\item[(d)]  Since $\fk s_4^c=0$, (d) holds vacuously.

 \end{enumerate}
This analysis suffices to conclude all but the lemma's final claim, 
%
 that $(\lambda - e_{1,2r+1}^*)(XY) = 0$ for all $X,Y \in \fk s_3$.  For this, recall that $\fk s_3 = \fk l_4 \oplus \fk l_4^c$.  By definition $\lambda(XY) =0$ for all $X \in \fk l_4$ and $Y \in \fk s_3$, and since $\fk s_3 = \fk s_4$, we also have by definition that $\lambda(XY)=0$ for all $X \in \fk l_4^c \subset \fk s_4$ and $Y \in \fk l_4$.  It is not difficult to see that the preceding sentence remains true if $\lambda$ is replaced by $e_{1,2r+1}^*$.  Thus if $X_1,Y_1 \in \fk l_4$ and $X_2,Y_2 \in \fk l_4^c$ and $X=X_1+X_2$ and $Y=Y_1+Y_2$ then 
 \[ (\lambda-e_{1,2r+1}^*)(XY) = (\lambda-e_{1,2r+1}^*)(X_2Y_2).\]
 Let $\fka = \fk l_4^c\oplus \FF_q \spanning\{e_{1,2r+1}\}$ as in (\ref{fka}); this is a subalgebra so  $X_2Y_2 \in \fka$, and the final part of the lemma follows by noting that $\lambda \downarrow \fka = e_{1,2r+1}^* \downarrow \fka$.
\end{proof}

We are now prepared to discuss the irreducible constituents of the character $\xi_\lambda$ is detail.  The following proposition proves almost all of Theorem \ref{main} by example.

\begin{proposition}\label{main-prop} Choose an integer $r>1$ and let $n>6r$.  If $\fkn = \fkt_{n}(q)$ and $\lambda \in \fkn^*$ is defined by (\ref{lambda}), then the following hold:

\begin{enumerate}
\item[(1)] The character $\xi_\lambda$ of $\UT_{n}(q)$ is the sum of $q^{r-1}$ distinct irreducible characters of degree $q^{5r^2-2r}$, each of which is induced from a linear character of $\ols_\lambda$.  Furthermore, 
\[\ba \olfkl_\lambda &=\left \{ X \in \fkn : X_\alpha = X_{\imap(\alpha)} \text{ if }\alpha \in \A\cup \B\cup \C  \text{ and }X_{\alpha} = 0\text{ if } \alpha  \in  \D \cup \Z \right\}, 
\\
\olfks_\lambda &= \left\{ X \in \fkn : X_\alpha = X_{\imap(\alpha)} \text{ if }\alpha \in \A\cup\B\cup\C\cup\D  \text{ and } X_{\alpha} = 0\text{ if } \alpha  \in  \Z  \right\}, \ea\] 
and $\xi_\lambda(1) =q^{5r^2-r-1}$ and $ \langle \xi_\lambda, \xi_\lambda \rangle_{\UT_n(q)} = q^{r-1}$.

\item[(2)] If $p>0$ is the characteristic of $\FF_q$ and $p^i$ is the largest power of $p$ less than or equal to $r$, then all irreducible constituents of $\xi_\lambda$ take values in $\QQ(\zeta_{p^{i+1}})$, but some irreducible constituents 
 have values which are not in $\QQ(\zeta_{p^i})$.

\item[(3)] The Kirillov functions $\psi_\lambda$ and  $\logpsi_\lambda$ have degree $q^{5r^2-2r}$, and  $\psi_\lambda$ is never a character of $\UT_n(q)$ while the exponential Kirillov function $\logpsi_\lambda$ is a character if and only if $r<p$.  

\end{enumerate}
\end{proposition}
 
\begin{remark}
A character of an algebra group is \emph{well-induced} if it is induced from a linear supercharacter of an algebra subgroup; see Section 4 in \cite{supp0}.  
One can adapt our arguments to prove that the Kirillov function $\psi_\mu$ is not a character for every $\mu\in \Xi_\lambda$, and given this, Proposition 4.1 in \cite{supp0} implies that none of the $q^{r-1}$ irreducible constituents of $\xi_\lambda$ are well-induced.  Moreover, one can show that all of our statements, except the descriptions of $\olfkl_\lambda$ and $\olfks_\lambda$ which change somewhat, hold verbatim if the $6r+1$ nonzero values of $\lambda_{ij}$ are replaced by arbitrary elements of $\FF_q^\times$, and that each of these $(q-1)^{6r+1}$ choices of $\lambda$ yields a distinct character $\xi_\lambda$.  For $r=2$  this gives rise to the $q(q-1)^{13}$  irreducible characters of $\UT_{13}(q)$ identified by Evseev in \cite[Theorem 2.7]{E} which are not well-induced. 
\end{remark}

\begin{proof}
By Observation \ref{inflation} we may assume without loss of generality that $n=6r+1$, since if $n>6r+1$ then we have a vector space decomposition $\fkn = \fkt_{6r+1}(q)\oplus \h$ where $\h$ is the two-sided ideal of matrices in $\fkn$ with zeros in the first $6r+1$ columns.

Let $\fk l = \olfkl_\lambda$ and $\fk s = \olfks_\lambda$  and  $L = \oll_\lambda$ and  $S = \ols_\lambda$ and $G = \UT_{n}(q)$.  Our descriptions of $\fk l$ and $\fk s$ are immediate from Lemma \ref{technical}.  Let $\mu = e_{1,2r+1}^* \downarrow \fk s$ and $\nu = (\lambda - e_{1,2r+1}^*) \downarrow \fk s$, so that $\lambda \downarrow \fk s  = \mu + \nu$.  By definition, $\xi_\lambda = \Ind_L^G( \theta_\lambda ) = \Ind_S^G (\chi_{\mu+\nu})$, and by Theorem \ref{structural} the irreducible constituents of $\xi_\lambda$ are in bijection with the irreducible constituents of the fully ramified supercharacter $\chi_{\mu+\nu}$ of $S$.    Also by Theorem \ref{structural}, we have $\langle \xi_\lambda,\xi_\lambda\rangle_G = |L|/|S| =|\fk s / \fk l|= q^{r-1}$ and  
$ \xi_\lambda(1) = |G|/|L|=|\fkn /\fk l| =q^{|\A| + |\B| +|\C| + |\D| +|\Z|} =  q^{5r^2-r-1}$ since 
\[\barr{c} |\A| = |\B| = |\C| = \frac{1}{2}r^2 - \frac{1}{2} r,\qquad |\D| =\frac{1}{2} r^2 +\frac{1}{2}r - 1,\qquad |\Z| = 3r^2.\earr\]

The last part of Lemma \ref{technical} states that $\nu(XY) = 0$ for all $X,Y \in \fk s$, and this implies that $g\nu h = \nu$ and $g(\mu+\nu)h = g \mu h + \nu$ for all $g,h \in S$.  Because of this property, it follows from the definition (\ref{superchar-def}) that $\chi_{\mu+\nu} = \chi_\mu \otimes \chi_\nu$ and $\psi_{\mu+\nu} = \psi_\mu \otimes \psi_{\nu}$, and that $\chi_\nu=\psi_\nu = \logpsi_\nu \in \Irr(S)$ is the linear character with the formula $\chi_\nu(g) = \theta\circ\nu(g-1)$ for $g \in S$.  
To prove the rest of the proposition, we decompose the supercharacter $\chi_\mu$ of $S$, using  Observation \ref{inflation} and the results of Section \ref{cmplx-chars}.
To this end, we observe that as a vector space $\fk s =  \fk a \oplus \h$ where 
\[\ba   
\fka &= \{ X \in \fk s : X_\alpha =0 \text{ if }\alpha \notin \D\text{ or }\alpha \neq (1,2r+1)\}, \\
\h &= \{ X \in \fk s : X_\alpha =0 \text{ if }\alpha \in \D\text{ or }\alpha = (1,2r+1)\}=\{ X \in \fk l : X_{1,2r+1} = 0 \}.
\ea
\]  We know that $\fk l$ is a two-sided ideal in $\fk s$, and it is easy to see that if $X \in \fk l$ and $Y \in \fk s$ then $(XY)_{1,2r+1} = (YX)_{1,2r+1} = 0$, since $Y_{i,2r+1} = 0$ for all $i>1$ and since $X_{i,2r+1} =0 $ whenever $Y_{1,i}\neq 0$.  Therefore $\h$ is also a two-sided ideal in $\fk s$.  Furthermore, it is clear that $\ker(\mu) \supset \h$.  

Now, as observed in Remark (iii) above, the vector space $\fka$ is a subalgebra naturally isomorphic to the algebra $\fka_{r+1}(q)$; 
under this isomorphism $\mu \downarrow \fka$ becomes identified with the functional $\kappa \in \fka_{r+1}(q)^*$ defined in Section \ref{cmplx-chars}.  Consequently, by Observation \ref{inflation} the irreducible constituents of $\chi_\mu$ are in bijection with  those of the supercharacter $\chi_\kappa$ of $\mathrm{A}_{r+1}(q)$, via a map of the form $\psi \mapsto \psi \circ \pi$ where $\pi : S \to \mathrm{A}_{r+1}(q)$ is some surjective homomorphism.  In particular, the characters on each side of this bijection have the same sets of values.
Since $\chi_\nu = \psi_\nu = \logpsi_\nu$ is linear with values in $\QQ(\zeta_p)$, our assertions in parts (2) and (3) thus follow by a combination of Observation \ref{inflation}, Proposition \ref{cmplx-constits}, and the remarks following Theorem \ref{structural}.
%
%
\end{proof}

As noted when defining $\lambda$, the character $\xi_\lambda$ is a constituent of the supercharacter of $\UT_n(q)$ whose shape is the set partition described in Theorem \ref{main}.  It remains to show that any supercharacter with the same shape has a constituent with  the same properties as $\xi_\lambda$.  This is immediate from the following observation, which proves Theorem \ref{main} in its entirety.

\begin{observation}\label{diagonal} The group of automorphisms of $\UT_n(q)$ of the form $g \mapsto D g D^{-1}$, where $D$ is a diagonal matrix in $\GL(n,\FF_q)$, acts transitively on the set of supercharacters of $\UT_n(q)$ with a given shape. 
\end{observation}

\begin{proof} Fix a diagonal matrix $D \in \GL(n,\FF_q)$ and let $\varphi_D$ be the conjugation map of $g\mapsto DgD^{-1}$.  From (\ref{superchar-def}) ones sees that if $\mu \in \fkt_n(q)^*$ then $\chi_\mu \circ \varphi_D = \chi_{\nu}$, where $\nu_{ij} = \frac{D_{ii}}{D_{jj}} \mu_{ij}$ for all $i,j \in [n]$.  
We may assume that $\mu$ is quasi-monomial (see the discussion in Section \ref{pattern}), and it is obvious that $\chi_\mu$ and $\chi_\nu$ have the same shape $\Lambda$.  Furthermore, it is not difficult to see that $\chi_\mu = \chi_\nu$ (which occurs if and only if $\mu=\nu$) if and only if $D_{ii} = D_{jj}$ whenever $i,j\in [n]$ belong to the same part of $\Lambda$.  
By the orbit-stabilizer theorem, the orbit of $\chi_\mu$ under the action of the diagonal matrices in $\GL(n,\FF_q)$ thus has cardinality $(q-1)^{n-\ell}$ where $\ell$ is the number of parts of $\Lambda $.  As this is precisely the number of quasi-monomial $\nu \in \fkt_n(q)^*$ with shape $\Lambda$, our statement follows.
\end{proof}

\end{document}